\documentclass[reqno,oneside,12pt]{amsart}

\usepackage{amsmath, amsfonts, amssymb,amsbsy,eucal,mathrsfs}

\usepackage{amsthm}

\usepackage[all]{xy}

\usepackage{chngcntr}

\usepackage{todonotes}

\usepackage{fullpage}

\numberwithin{equation}{section}
\newtheorem{thm}{Theorem}[section]

\newtheorem{lemma}[thm]{Lemma}
\newtheorem{prop}[thm]{Proposition}
\newtheorem{cor}[thm]{Corollary}

\newtheorem{rmk}[thm]{Remark}

\theoremstyle{remark}

\newtheorem{example}[thm]{Example}
\newtheorem{defi}[thm]{Definition}

\DeclareMathOperator{\G}{G}    
\DeclareMathOperator{\IG}{IG}
\DeclareMathOperator{\Fl}{Fl}
\DeclareMathOperator{\Sp}{Sp}
\DeclareMathOperator{\QH}{QH}
\DeclareMathOperator{\BQH}{BQH}

\DeclareMathOperator{\Ext}{Ext}

\DeclareMathOperator{\Specan}{Specan}
\DeclareMathOperator{\Spec}{Spec}
\DeclareMathOperator{\Lie}{Lie}

\DeclareMathOperator{\Fuk}{\sf{Fuk}}
\DeclareMathOperator{\FS}{\sf{FS}}
\DeclareMathOperator{\MF}{\sf{MF}}
\DeclareMathOperator{\Coh}{\sf{Coh}}

\newcommand{\D}{{\Delta}}
\newcommand{\ZZ}{{\mathbb Z}}
\newcommand{\QQ}{{\mathbb Q}}
\newcommand{\bQ}{{\mathbb Q}}

\newcommand{\bC}{{\mathbb C}}
\newcommand{\starz}{{\!\ \star_{0}\!\ }}
\newcommand{\pt}{{{\rm pt}}}
\newcommand{\C}{{\mathbb C}}
\newcommand{\p}{\mathbb{P}}

\def \Yo {{\mathring{Y}}}
\def \spec {{{\rm Spec}}}

\newcommand{\cA}{\mathcal{A}}
\newcommand{\cC}{\mathcal{C}}

\newcommand{\cO}{\mathcal{O}}
\newcommand{\cU}{\mathcal{U}}
\newcommand{\cB}{\mathcal{B}}
\newcommand{\cT}{\mathcal{T}}
\newcommand{\cV}{\mathcal{V}}

\newcommand{\cR}{\mathcal{R}}
\newcommand{\cIA}{{\bar{\mathcal{A}}}}
\newcommand{\cIR}{{\bar{\mathcal{R}}}}

\newcommand{\LL}{\mathbb{L}}

\begin{document}

\begin{center}
\textbf{On quantum cohomology of Grassmannians of isotropic lines,\\ unfoldings of $A_n$-singularities, and Lefschetz exceptional collections}\\
\vspace{7pt}
J.~A.~Cruz~Morales, A.~Mellit, N.~Perrin, M.~Smirnov \\
\vspace{5pt}
with an appendix by A.~Kuznetsov\\
\end{center}

\vspace{20pt}

\begin{flushright}\textit{
Dedicated to Yuri~Ivanovich~Manin, \\
with admiration and gratitude}
\end{flushright}

\vspace{15pt}

{\small
\noindent {\scshape Abstract.} The subject of this paper is the big quantum cohomology rings of symplectic isotropic Grassmannians $\IG(2, 2n)$. We show that these rings are regular. In particular, by~``generic smoothness", we obtain a conceptual proof of generic semisimplicity of the big quantum cohomology for $\IG(2, 2n)$. Further, by a general result of C.~Hertling, the regularity of these rings implies that they have a description in terms of isolated hypersurface singularities, which we show in this case to be of type~$A_{n-1}$. By the homological mirror symmetry conjecture, these results suggest the existence of a very special full exceptional collection in the derived category of coherent sheaves on $\IG(2, 2n)$. Such a collection is constructed in the appendix by Alexander Kuznetsov.
}

\author{J.~A.~Cruz~Morales}
\address{Universidad Nacional de Colombia, Departamento de Matem\'aticas,
Carrera 45 No. 26--85  Edificio Uriel Guti\'errez
Bogot\'a D.C.,  Colombia}
\email{jacruzmo@unal.edu.co}

\author{A.~Mellit}
\address{Institute of Science and Technology, Austria, Am Campus 1, 3400 Klosterneuburg, Austria}
\email{mellit@gmail.com}

\author{N.~Perrin}
\address{Laboratoire de Math\'ematiques de Versailles, UVSQ, CNRS, Universit\'e Paris--Saclay, 78035 Versailles, France}
\email{nicolas.perrin@uvsq.fr}

\author{M.~Smirnov}
\address{
Universit\"at Augsburg,
Institut f\"ur Mathematik,
Universit\"atsstr.~14,
86159 Augsburg,
Germany
}
\email{maxim.smirnov@math.uni-augsburg.de}

\author{A.~Kuznetsov}
\address{
\parbox{0.95\textwidth}{
Algebraic Geometry Section, Steklov Mathematical Institute of Russian Academy of Sciences,
8 Gubkin str., Moscow 119991 Russia
\\[3pt]
The Poncelet Laboratory, Independent University of Moscow
\\[3pt]
Laboratory of Algebraic Geometry, National Research University Higher School of Economics
\bigskip
}
} 

\email{akuznet@mi.ras.ru}

\vspace{10pt}

\tableofcontents

\section{Introduction}

\subsection{Dubrovin's conjecture}

Quantum cohomology rings of smooth projective varieties have been an object of intensive study ever since they were introduced at the beginning of 1990s. A particular question that attracted a lot of attention is a conjecture formulated by Boris Dubrovin in \cite{Du}. This conjecture provides a beautiful relation between two \textit{a priori} seemingly unrelated objects --- the quantum cohomology ring of a variety $X$ and its bounded derived category of coherent sheaves $D^b(X)$. Dubrovin's conjecture has several parts, and its first part claims that the generic semisimplicity of the quantum cohomology of $X$ is equivalent to the existence of a full exceptional collection in $D^b(X)$. Though there are no general approaches to this conjecture, it has been tested in many examples and the original formulation of Dubrovin has been made more precise (see \cite{Ba,Du, Du-Strasbourg,GaGoIr,HeMaTe} and references therein).

\subsection{Quantum cohomology: big vs. small}

The quantum cohomology ring of $X$ is a (formal) deformation family of rings provided by the genus zero Gromov--Witten theory, that specializes to the ordinary cohomology ring $H^*(X,\bQ)$, if one sets all deformation parameters to zero. There exist two such deformation families that one often meets in the literature: the small quantum cohomology (involves only 3-point GW invariants) and the big quantum cohomology (involves GW invariants with arbitrary many insertions). The former family is a specialization of the latter (i.e. some of the deformation parameters are set to zero). It is the big quantum cohomology ring that Dubrovin used in \cite{Du} to formulate the conjecture.

In this paper we are concerned with generic semisimplicity of quantum cohomology. Therefore, we only consider the above deformation families in the neighbourhood of the generic point of the small quantum cohomology. In this way the small quantum cohomology (denote it $\QH(X)$) becomes a commutative, associative, finite dimensional algebra over some field $K$, and the big quantum cohomology (denote it $\BQH(X)$) is a formal deformation family of $\QH(X)$. This will be made more precise in Section~\ref{Sec.: Conventions and Notation}.

In Dubrovin's conjecture we are interested in the semisimplicity of the generic member of the deformation family $\BQH(X)$. So it could happen that the special fiber $\QH(X)$ is not semisimple, whereas the generic fiber of the full family $\BQH(X)$ is. This is exactly what happens in the case of symplectic isotropic Grassmannians $\IG(2,2n)$. Indeed, it was shown in \cite{ChMaPe,ChPe} that the small quantum cohomology of $\IG(2,2n)$ is not semisimple. Further, in \cite{GMS, Pe} jointly with Sergey Galkin, we have proved that the big quantum cohomology of $\IG(2,2n)$ is generically semisimple. Since, according to \cite{Ku, Sa}, the derived category of coherent sheaves on $\IG(2,2n)$ has a full exceptional collection, the first part of Dubrovin's conjecture holds.

\subsection{Generic semisimplicity via generic smoothness}

Our first result is the following (see Theorem~\ref{Prop.: Regularity of BQH}).

\medskip

\noindent {\bf Theorem~A.} The ring $\BQH(\IG(2,2n))$ is regular.

\medskip

As an easy consequence of the above theorem and generic smoothness (see Corollary~\ref{Thm.: semisimplicity}) we recover generic semisimplicity of the big quantum cohomology of $\IG(2,2n)$ proved in \cite{GMS, Pe}. Note that in~\cite{GMS} the authors only considered the case of $\IG(2,6)$ and the proof was computer-assisted. The proof in \cite{Pe} works for all $\IG(2,2n)$ but needs many lengthy computations. Our proof is more conceptual and clarifies the situation. This is the content of Sections \ref{Sec.: Conventions and Notation}--\ref{Sec.: BQH}.

Our approach via generic smoothness gives a new perspective towards a proof of generic semisimplicity of quantum cohomology for more general Grassmannians $\IG(m,2n)$ or even rational homogeneous spaces $G/P$. We plan to address this in a future work.

\subsection{$F$-manifolds with smooth spectral cover}

In \cite[Problem 2.8]{HeMaTe} the authors formulated the following question: {\it characterise those varieties $X$ for which the quantum cohomology has smooth spectral cover}. In our terminology this question translates into the problem of characterising those varieties $X$ for which the big quantum cohomology $\BQH(X)$ is a regular ring. According to the above theorem, isotropic Grassmannians $\IG(2,2n)$ provide non-trivial examples. In fact, it is natural to expect that this property holds for any rational homogeneous space $G/P$. 

The regularity of $\BQH(X)$ implies, by a beautiful result of C.~Hertling, that the $F$-manifold defined by the quantum cohomology of $X$ has a description in terms of unfoldings of isolated hypersurface singularities. In our situation we prove the following (see Section~\ref{SubSec.: F-mfd IG(2,2n)}).

\medskip

\noindent {\bf Theorem~B.} Assume that the genus zero Gromov--Witten potential of $\IG(2,2n)$ has non-trivial convergence radius. Then the $F$-manifold of $\BQH(\IG(2,2n))$ decomposes into the product of the unfolding of an $A_{n-1}$-singularity and $(2n-1)(n-1)$ copies of the unfolding of an $A_1$-singularity.
\medskip

\subsection{Derived category of coherent sheaves}

The above theorem and mirror symmetry suggest that $\IG(2,2n)$ should have a Landau--Ginzburg model with one degenerate critical point of type~$A_{n-1}$ and $(2n-1)(n-1)$ non-degenerate critical points. Therefore, the Fukaya--Seidel category of this LG model should have a semiorthogonal decomposition 
\begin{align*}
\langle \cC_{n-1}, E_1 , \dots , E_{(2n-1)(n-1)}  \rangle,
\end{align*} 
where $E_i$ are exceptional objects given by the non-degenerate critical points and $\cC_{n-1}$ is the Fukaya--Seidel category of an $A_{n-1}$-singularity. By homological mirror symmetry conjecture, the bounded derived category of coherent sheaves $D^b(\IG(2,2n))$ should be equivalent to the Fukaya--Seidel category of the LG model. Hence, the derived category $D^b(\IG(2,2n))$ should also have a decomposition of this form. A more detailed account is contained in Section~\ref{Sec.: LG model}.

The above discussion is highly conjectural. Nonetheless, the aforementioned semiorthogonal decomposition of $D^b(\IG(2,2n))$ is constructed directly in the Appendix by Alexander Kuznetsov.  

\medskip

\noindent {\bf Theorem~C (Kuznetsov).} There exists a semiorthogonal decomposition
\begin{align*}
D^b (\IG(2,2n)) = \langle \cA_{n-1}, E_1 , \dots , E_{(2n-1)(n-1)}  \rangle,
\end{align*} 
where the $E_i$ are some exceptional objects and the subcategory $\cA_{n-1}$ is equivalent to the bounded derived category of representations of the quiver of type $A_{n-1}$.

\medskip

This theorem confirms the conjectural picture described above, as the bounded derived category of representations of the quiver of type $A_{n-1}$ is equivalent to the Fukaya--Seidel category of an $A_{n-1}$-singularity by \cite{Sei}. Note that the result in the appendix is stronger. Namely, the objects $E_1 , \dots , E_{(2n-1)(n-1)}$ form a rectangular Lefschetz exceptional collection.

One can view these results as an enhanced version of Dubrovin's conjecture, i.e. a prediction of a more subtle relation between quantum cohomology rings and derived categories of coherent sheaves.

\subsection*{Acknowledgements}

We are indebted to Yuri Ivanovich Manin for his support over the years and for drawing our attention to results of Claus~Hertling indispensable for this paper. Further we are very grateful to Alexander Kuznetsov for providing the appendix, and numerous remarks on the main body of the paper. Special thanks go to Mohammed Abouzaid, Sheel Ganatra, and Ailsa Keating for the helpful email correspondence about Fukaya--Seidel categories. Last but not least we would like to thank our friends and colleagues Tarig Abdelgadir, Erik Carlsson, Roman Fedorov, Sergey Galkin, Christian Lehn, Sina T\"ureli and Runako Williams for valuable discussions and comments.

We would like to thank institutions that supported us at various stages of this project. Namely, we are very grateful to the International Centre for Theoretical Physics (ICTP) in Trieste, the Institute for Algebraic Geometry and the Riemann Center for Geometry and Physics at the Leibniz Universit\"at Hannover, and the Max Planck Institute for Mathematics (MPIM) in Bonn. The third author was supported by a public grant as part of the Investissement d'avenir project, reference ANR-11-LABX-0056-LMH, LabEx LMH.

\section{Conventions and notation for quantum cohomology}
\label{Sec.: Conventions and Notation}

Here we briefly recall the definition of the quantum cohomology ring for a smooth projective variety $X$. To simplify the exposition and avoid introducing unnecessary notation we impose from the beginning the following conditions on $X$: it is a Fano variety of Picard rank~1 and $H^{odd}(X,\QQ)=0$. For a thorough introduction we refer to \cite{Ma}.

\subsection{Definition}
\label{SubSec.: Def of QH}

Let us fix a graded basis $\Delta_0, \dots , \Delta_s$ in $H^*(X, \QQ)$ and dual linear coordinates $t_0, \dots, t_s$. It is customary to choose $\Delta_0=1$. For cohomology classes we use the Chow grading, i.e. we divide the topological degree by two. Further, for variables $t_i$ we set $\deg(t_i)=1-\deg(\D_i)$.

Let $R$ be the ring of formal power series $\QQ[[q]]$, $k$ its field of fractions, and $K$ an algebraic closure of $k$. We set  $\deg(q)= \text{index} \, (X)$, which is the largest integer $n$ such that $-K_X = nH$ for some ample divisor $H$ on $X$, where $K_X$ is the canonical class of $X$.

The genus zero Gromov--Witten potential of $X$ is an element $\Phi \in R[[t_0, \dots, t_s]]$ defined by the formula
\begin{align}\label{Eq.: GW potential}
&\Phi = \sum_{(i_0, \dots , i_s)}  \langle \Delta_0^{\otimes i_0}, \dots, \Delta_s^{\otimes i_s} \rangle \frac{t_0^{i_0} \dots t_s^{i_s} }{i_0!\dots i_s!}, 
\end{align}
where 
\begin{align*}
\langle \Delta_0^{\otimes i_0}, \dots, \Delta_s^{\otimes i_s} \rangle = \sum_{d=0}^{\infty} \langle \Delta_0^{\otimes i_0}, \dots, \Delta_s^{\otimes i_s} \rangle_d q^d,
\end{align*}
and $\langle \Delta_0^{\otimes i_0}, \dots, \Delta_s^{\otimes i_s} \rangle_d$ are rational numbers called Gromov--Witten invariants of $X$ of degree $d$. With respect to the grading defined above $\Phi$ is homogeneous of degree $3 - \dim X$.

Using \eqref{Eq.: GW potential} one defines the \textit{big quantum cohomology ring} of $X$. Namely, let us endow the $K[[t_0, \dots, t_s]]$-module
\begin{align*}
\BQH(X) = H^*(X, \bQ)\otimes_{\bQ} K[[t_0, \dots, t_s]]
\end{align*}
with a ring structure by setting
\begin{align}\label{Eq.: Quantum cohomology}
\Delta_a \star \Delta_b = \sum_c \frac{\partial^3 \Phi}{\partial t_a \partial t_b \partial t_c} \Delta^c,
\end{align}
on the basis elements and extending to the whole $\BQH(X)$ by $K[[t_0, \dots, t_s]]$-linearity. Here  $\D^0, \dots, \D^s$ is the basis dual to  $\D_0, \dots, \D_s$ with respect to the Poincar\'e pairing. It is well known that~\eqref{Eq.: Quantum cohomology} makes $\BQH(X)$ into a commutative, associative, graded $K[[t_0, \dots, t_s]]$-algebra with the identity element $\D_0$.

The algebra $\BQH(X)$ is called the \textit{big} quantum cohomology algebra of $X$ to distinguish it from a simpler object called the \textit{small} quantum cohomology algebra which is the quotient of $\BQH(X)$ with respect to the ideal $(t_0, \dots, t_s)$. We will denote the latter $\QH(X)$ and use $\starz$ instead of $\star$ for the product in this algebra. It is a finite dimensional $K$-algebra. Equivalently one can say that 
$$\QH(X)=H^*(X,\bQ) \otimes_{\bQ} K$$ 
as a vector space, and the $K$-algebra structure is defined by putting 
$$\Delta_a \starz \Delta_b = \sum_c \langle \D_a, \D_b, \D_c \rangle \Delta^c.$$

\begin{rmk}
We are using a somewhat non-standard notation $\BQH(X)$ for the big quantum cohomology and $\QH(X)$ for the small quantum cohomology to stress the difference between the two. Note that this notation is different from the one used in \cite{GMS} and is closer to the notation of \cite{Pe}.
\end{rmk}

\begin{rmk}\label{SubSubSec.: Remark on difference of notation with Manin's book}

The above definitions look slightly different from the ones given in~\cite{Ma}. The differences are of two types. The first one is that $\QH(X)$ and $\BQH(X)$ are in fact defined already over the ring $R$ and not only over $K$. We pass to $K$ from the beginning, since in this paper we are only interested in generic semisimplicity of quantum cohomology. The second difference is that in some papers on quantum cohomology one unifies the coordinate $q$ with the coordinate $t_i$ which is dual to $H^2(X, \bQ)$, but the resulting structures are equivalent.
\end{rmk}

\subsection{Deformation picture}
\label{SubSec.: Def. Picture}

The small quantum cohomology, if considered over the ring $R$ (cf. Remark~\ref{SubSubSec.: Remark on difference of notation with Manin's book}), is a deformation of the ordinary cohomology algebra, i.e. if we put $q=0$, then the quantum product becomes the ordinary cup-product. Similarly, the big quantum cohomology is an even bigger deformation family of algebras. Since we work not over $R$ but over $K$, we lose the point of classical limit but still retain the fact that $\BQH(X)$ is a deformation family of algebras with the special fiber being $\QH(X)$. 

In this paper we view $\spec(\BQH(X))$ as a deformation family of zero-dimensional schemes over $\spec (K[[t_0,\dots, t_s]])$. In the base of the deformation we consider the following two points: the origin (the closed point given by the maximal ideal $(t_0, \dots , t_s)$) and the generic point $\eta$. The fiber of this family over the origin is the spectrum of the small quantum cohomology $\spec(\QH(X))$. The fiber over the generic point will be denoted by $\spec(\BQH(X)_{\eta})$. It is convenient to summarize this setup in the diagram
\begin{align}\label{Eq.: BQH family}
\vcenter{
\xymatrix{
\spec(\QH(X)) \ar[d] \ar[r]   &    \spec(\BQH(X)) \ar[d]^{\pi} & \ar[l] \spec(\BQH(X)_{\eta}) \ar[d]^{\pi_{\eta}} \\
\spec (K)    \ar[r]      & \spec (K[[t_0, \dots, t_s]])  &   \ar[l] \eta
}
}
\end{align}
where both squares are Cartesian. 
    
By construction $\BQH(X)$ is a free module of finite rank over $K[[t_0, \dots,t_s]]$. Therefore, it is a noetherian semilocal $K$-algebra which is flat and finite over $K[[t_0, \dots,t_s]]$. Note that neither $K[[t_0, \dots,t_s]]$ nor $\BQH(X)$ are finitely generated over the ground field $K$. Therefore, some extra care is required in the standard commutative algebra (or algebraic geometry) constructions. For example, the notion of smoothness is one of such concepts.

\subsection{Semisimplicity}
\label{SubSec.: Semisimplicity}

Let $A$ be a finite dimensional algebra over a field $F$ of characteristic zero. It is called \textit{semisimple} if it is a product of fields. Equivalently, the algebra $A$ is semisimple if the scheme $\spec(A)$ is reduced. Another equivalent condition is to require the morphism $\spec(A) \to \spec(F)$ to be smooth.

\begin{defi}
We say that $\BQH(X)$ is \textit{generically semisimple} if $\BQH(X)_{\eta}$ is a semisimple algebra.
\end{defi}

\section{Geometry of $\IG(2,2n)$}
\label{Sec.: Geometry}

Let $V$ be a complex vector space endowed with a symplectic form $\omega$. In this case the dimension of $V$ has to be even and we denote it by $2n$. For any $1 \leq m \leq n$ there exists an algebraic variety $\IG(m, V)$ that parametrizes $m$-dimensional isotropic subspaces of $V$. For $m=2$, and this is the case we are considering in this paper, it has the following explicit description. Consider the ordinary Grassmannian $\G(2,V)$ with its Pl\"ucker embedding into $\mathbb{P}(\Lambda^2V)$. The symplectic form $\omega$ defines a hyperplane $H_{\omega}=\p \, (\text{Ker} \,\omega) \subset \mathbb{P}(\Lambda^2V)$ and the intersection of $\G(2,V)$ with $H_{\omega}$ is  $\IG(2, V)$. Thus, we have inclusions
\begin{align}\label{Eq.: Hyperplane section diagram for IG(2,2n)}
\IG(2, V) \subset \G(2,V) \subset  \mathbb{P}(\Lambda^2V).
\end{align}
As for different symplectic forms on $V$ we obtain isomorphic isotropic Grassmannians, it is unambiguous to write $\IG(2,2n)$. We will use this notation starting from Paragraph~\ref{SubSec.: Cohomology of IG}.

\subsection{Two sets of cohomology classes}

As for ordinary Grassmannians, one considers the short exact sequence of vector bundles on $X=\IG(2, V)$
\begin{align}\label{Eq.: Tautological S.E.S I}
0 \to \mathcal{U} \to \mathcal{V} \to \mathcal{V/U} \to 0,
\end{align}
where $\mathcal{V}$ is the trivial vector bundle with fiber $V$, $\mathcal{U}$ is the subbundle of isotropic subspaces, and $\mathcal{V/U}$ is the quotient bundle. Usually one refers to $\mathcal{U}$ and $\mathcal{V/U}$ as \textit{tautological subbundle} and \textit{tautological quotient bundle} respectively. 

One also defines a vector bundle $\cU^\perp$ as the kernel of the composition $\cV \overset{\omega}{\to} \cV^* \to \cU^*$, where the first morphism is the isomorphism induced by the symplectic form, and the second one is the dual of the natural inclusion $\cU \to \cV$. From the definition of $\cU^\perp$ we immediately obtain an isomorphism
\begin{align}\label{Eq.: V/U^perp iso U^*}
\cV/\cU^\perp \simeq \cU^*.
\end{align}

Further, we have inclusions of vector bundles $\cU \subset \cU^\perp \subset \cV$, and can also consider the short exact sequence
\begin{align}\label{Eq.: Tautological S.E.S II}
0 \to \mathcal{U}^{\perp}/\mathcal{U} \to \mathcal{V}/\mathcal{U} \to \mathcal{V}/\mathcal{U}^\perp \to 0.
\end{align}
By taking Chern classes of vector bundles in \eqref{Eq.: Tautological S.E.S I} and \eqref{Eq.: Tautological S.E.S II} we obtain two sets of cohomology classes which generate the cohomology ring:  

\smallskip

\noindent \textbf{(a)~Chern classes of $\mathcal{V/U}$.} The vector bundle $\cV/\cU$ is of rank $2n-2$ and and we denote its Chern classes by $\sigma_k = c_k(\cV/\cU)$. Cycles representing these cohomology classes can be explicitly described as follows. Let
\begin{align}\label{Eq.: Schubert subvariety}
Z(E_{2n - k - 1}) = \{V_2 \in X \ | \ \dim(V_2 \cap E_{2n - k - 1}) \geq 1 \}.
\end{align}
Then $\sigma_k = [Z(E_{2n - k - 1})] \in H^{2k}(X,\ZZ)$. In the above, $E_{2n - k - 1}$ is a subspace of dimension $2n - k -1$ such that the rank of $\omega\vert_{E_{2n - k - 1}}$ is minimal. Note that $\sigma_0 = 1$ is the fundamental class of $X$. These cohomology classes are usually called \textit{special Schubert classes}.

\smallskip

\noindent \textbf{(b)~Chern classes of $\mathcal{U}$ and $\mathcal{U}^{\perp}/\mathcal{U}$.} The vector bundle $\mathcal{U}$ is of rank $2$, so it only has two non-vanishing Chern classes $a_i=c_i(\mathcal{U})$ for $i=1, 2$. The vector bundle $\mathcal{U}^{\perp}/\mathcal{U}$ is self-dual of rank $2n-4$, therefore it has only $n-2$ non-vanishing Chern classes $b_i= c_{2i}(\mathcal{U}^{\perp}/\mathcal{U})$ for $i\in [1, n- 2]$.

\subsection{Cohomology ring of $\IG(2,2n)$}
\label{SubSec.: Cohomology of IG}

The cohomology ring of $X=\IG(2,2n)$ can be described in terms of generators and relations. We will give two presentations using the two sets of special cohomology classes defined above.

\begin{prop}[\cite{BKT}, Theorem 1.2]
\label{Prop.: Presentation for coh I}

The cohomology ring $H^*(X,\QQ)$ is isomorphic to the quotient of the ring $\QQ[\sigma_1,\cdots,\sigma_{2n-2}]$ by the ideal generated by the elements
\begin{align}\label{Eq.: Eqns for H(IG), part I}
\det(\sigma_{1 + j - i})_{1 \leq i,j \leq r}, \ \ \ \textrm{with } r \in [3,2n-2]
\end{align}
and the two elements 
\begin{align}\label{Eq.: Eqns for H(IG), part II}
\sigma_{n-1}^2 + 2 \sum_{i = 1}^{n - 1} (-1)^i \sigma_{n - 1 + i} \sigma_{n - 1 - i}  \ \ \textrm{  and  }\ \  \sigma_n^2 + 2 \sum_{i = 1}^{n - 2} (-1)^i \sigma_{n + i} \sigma_{n - i}.
\end{align}
The dimension of $H^*(X, \QQ)$ is equal to $2^2\binom{n}{2}=2n(n-1)$.
\end{prop}

\begin{prop} 
\label{Prop.: Presentation for coh II}
The cohomology ring $H^*(X,\QQ)$ is isomorphic to the quotient of the ring $\QQ[a_1,a_2,b_1,\cdots,b_{n-2}]$ by the ideal generated by
\begin{align}\label{Eq.: Relation in Presentation II}
(1 + (2a_2 - a_1^2) x^2 + a_2^2 x^4)(1 + b_1 x^2 + \cdots + b_{n-2} x^{2n-4}) = 1
\end{align}
The last equality is viewed as an equality of polynomials in the variable $x$ and gives a concise way to write a system of  equations in the variables $a_i, b_i$.
\end{prop}

\begin{proof} 

This result is well know to specialists but we include a short proof for the convenience of the reader.

Let us start by checking that \eqref{Eq.: Relation in Presentation II} holds in the cohomology ring. Define $P(x)= 1 + a_1 x + a_2 x^2$ and $Q(x)= 1 + b_1 x^2 + \dots + b_{n-2}x^{2n-4}$ and rewrite \eqref{Eq.: Relation in Presentation II} as
$
P(x)P(-x)Q(x)=1.
$
We interpret the polynomial $P(x)$ as the total Chern class of $\mathcal{U}$ and $Q(x)$ as the total Chern class of $\mathcal{U}^{\perp}/\mathcal{U}$. Now by using basic properties of Chern classes, short exact sequences \eqref{Eq.: Tautological S.E.S I},\eqref{Eq.: Tautological S.E.S II}, and  the isomorphism \eqref{Eq.: V/U^perp iso U^*}, it is easy to see that the above relation does hold. 

The above discussion shows that we have a natural homomorphism of $\QQ$-algebras
\begin{align}\label{Eq.: algebra hom in the proof of pres II}
\psi \colon \QQ[a_1,a_2,b_1,\cdots,b_{n-2}]/(P(x)P(-x)Q(x) -1)  \to H^*(X, \QQ)
\end{align}   
sending $a_i$'s to $a_i$'s and $b_i$'s to $b_i$'s. To prove the proposition it is enough to establish two facts: i) $\sigma_i$'s can be expressed in terms of $a_i$'s and $b_i$'s, ii) the dimensions of both algebras in \eqref{Eq.: algebra hom in the proof of pres II} are equal.  To prove the first fact one can use again simple properties of Chern classes and the exact sequence \eqref{Eq.: Tautological S.E.S II}. For the second fact, we first need to show that $\QQ[a_1,a_2,b_1,\cdots,b_{n-2}]/(P(x)P(-x)Q(x) -1)$ is finite dimensional. This can be done similarly to the finite-dimensionality part of the proof of \cite[Theorem 1.2]{BKT}. Then, we need to compute the dimension of this algebra. For this we can proceed similarly to the proof of \cite[Lemma 1.1]{BKT}, which is based on \cite{S}, and get the desired dimension $2n(n-1)$. 
\end{proof}

\section{Small quantum cohomology of $\IG(2,2n)$}

\subsection{Two presentations}

As described in Section \ref{SubSec.: Cohomology of IG}, the ordinary cohomology ring of the Grassmannian $\IG(2,2n)$ is generated by the special Schubert classes $\sigma_1, \dots , \sigma_{2n-2}$ with relations \eqref{Eq.: Eqns for H(IG), part I} and \eqref{Eq.: Eqns for H(IG), part II}. To pass to the small quantum cohomology, informally speaking, we need to adjoin a new variable $q$ to the $\sigma_i$'s and adjust relations \eqref{Eq.: Eqns for H(IG), part I}--\eqref{Eq.: Eqns for H(IG), part II} in such a way that they remain homogeneous and specialize to the original ones when setting $q=0$. The degree of the variable $q$ in the case of $\IG(2,2n)$ is equal to $2n-1$ (cf. Section~\ref{SubSec.: Def of QH}). Thus, the only relation that needs to be modified is the second equation in \eqref{Eq.: Eqns for H(IG), part II}. Moreover, up to a constant factor, this modification is unique for degree reasons. The complete answer for arbitrary Grassmannians $\IG(m,2n)$ was given in \cite[Theorem 1.5]{BKT} which we reproduce here in the special case of $m=2$.

\begin{thm}[\cite{BKT}, Theorem 1.5]
\label{thm-bkt}
The small quantum cohomology ring $\QH(X)$ is isomorphic to the quotient of the ring $K[\sigma_1,\cdots,\sigma_{2n-2}]$ by the ideal generated by the elements
$$\det(\sigma_{1 + j - i})_{1 \leq i,j \leq r}, \ \ \ \textrm{with } r \in [3,2n-2]$$
and the two elements 
$$\sigma_{n-1}^2 + 2 \sum_{i = 1}^{n - 1} (-1)^i \sigma_{n - 1 + i} \sigma_{n - 1 - i}  \ \ \textrm{  and  }\ \  \sigma_n^2 + 2 \sum_{i = 1}^{n - 2} (-1)^i \sigma_{n + i} \sigma_{n - i} + (-1)^{n+1} q \sigma_{1}.$$
\end{thm}

\medskip

Combining the above theorem with Proposition \ref{Prop.: Presentation for coh II} we arrive at the following statement.
\begin{cor}\label{Cor.: Presentation for small QH}
The small quantum cohomology ring $\QH(X)$ is isomorphic to the quotient of the ring $K[a_1,a_2,b_1,\cdots,b_{n-2}]$ by the ideal generated by 
\begin{align}\label{Eq.: Relations for small QH II}
(1 + (2a_2 - a_1^2) x^2 + a_2^2 x^4)(1 + b_1 x^2 + \cdots + b_{n-2}
  x^{2n-4}) = 1 - q a_1 x^{2n}.
\end{align}

\begin{proof}
 
 We need to find the quantum deformation of \eqref{Eq.: Relation in Presentation II}. Since the index of $X$ is $2n-1$, there are no quantum corrections except possibly in degrees $2n-1$ and $2n$. Furthermore, as  \eqref{Eq.: Relation in Presentation II} has no terms of degree $2n-1$, we only need to check the deformation in degree $2n$. Explicitly we need to check the relation $a_2^2b_{n-2} = -qa_1$. Note that since $a_2^2b_{n-2}$ has degree $2n$ and since $a_2^2b_{n-2}$ vanishes in $H^*(X,\QQ)$ (see Proposition \ref{Prop.: Presentation for coh II}), we have $a_2^2b_{n-2} = \lambda a_1 q$ with $\lambda = - \langle a_2,a_2b_{n-2},\ell \rangle_1$ and $\ell$ the class of a line in $X$.
 
By the definition of $a_i$'s we have $a_2 = c_2(\mathcal{U})$, and so $a_2$ is a Schubert class (in fact, we have $a_2 = \sigma_{1,1}$). Further, using basic properties of Chern classes, and Formulas \eqref{Eq.: Tautological S.E.S I}--\eqref{Eq.: Tautological S.E.S II} one obtains the equality of cohomology classes
$$
a_2b_{n-2}  = \sigma_{2n-2}.
$$
Therefore, $a_2b_{n-2}$ is also a Schubert class. As $\ell$ is yet again a Schubert class, all insertions in the GW invariant $\langle a_2,a_2b_{n-2},\ell \rangle_1$ are Schubert classes in $X$. One can compute such an invariant either by using methods of Section \ref{sec:4points}, or the quantum-to-classical principle for degree $1$ invariants from \cite[Section 4]{ChPe}, and obtains $\langle a_2,a_2b_{n-2},\ell \rangle_1 = 1$.
\end{proof}

\end{cor}

\subsection{Structure of $\QH(X)$}
\label{SubSec.: Structure of QH}

In this paragraph we decompose the $K$-algebra $\QH(X)$ as a direct product, or, equivalently, we decompose the $K$-scheme $\spec(\QH(X))$ into connected components.

\begin{prop}\label{Prop.: Decomposition of QH}
The scheme $\spec (\QH(X))$ is the disjoint union of $(2n-1)(n-1)$ reduced points $\spec(K)$ and one fat point $\spec(K[\varepsilon]/(\varepsilon^{n-1}))$.
\end{prop}

\begin{proof}
By the presentation of $\QH(X)$ described in Corollary \ref{Cor.: Presentation for small QH}, the $K$-scheme $\spec(\QH(X))$ is given as a closed subscheme of the affine space $\mathbb{A}^n =\spec (K[a_1,a_2, b_1, \dots, b_{n-2}])$ defined by the equations
\begin{align}\label{Eq.: System for QH(IG(2,2n))} \notag
& 2a_2 -a_1^2 + b_1 = 0 \\ \notag
& a_2^2 + b_1(2a_2 -a_1^2)+b_2 =0 \\ \notag
& b_1a_2^2 + b_2(2a_2 -a_1^2)+b_3 =0 \\ \notag
& \dots \\
& b_{i-1} a_2^2 + b_i(2a_2 -a_1^2) + b_{i+1} =0 \\ \notag
& \dots \\ \notag
& b_{n-4} a_2^2 + b_{n-3}(2a_2 -a_1^2) + b_{n-2} =0 \\ \notag
& b_{n-3} a_2^2 + b_{n-2}(2a_2 -a_1^2)  =0 \\ \notag
& b_{n-2}a_2^2 + qa_1 =0 .
\end{align}
It is clear that the origin of $\mathbb{A}^n$ is a solution of this system. Moreover, this solution corresponds to a fat point of $\spec(\QH(X))$. Indeed, it is easy to see that the Zariski tangent space of \eqref{Eq.: System for QH(IG(2,2n))} at the origin is one-dimensional. Thus, the origin is a fat point of $\spec(\QH(X))$. Let $A$ be the corresponding factor of $\QH(X)$. Thus, we have the direct product decomposition
\begin{align}\label{Eq.: Direct product decomposition for QH}
\QH(X) = A \times B,
\end{align}
where $B$ corresponds to components of $\spec(\QH(X))$ supported outside the origin.

Setting $a_1=0$ in \eqref{Eq.: System for QH(IG(2,2n))} we obtain $b_i=(-1)^i i a_2^i$. Therefore, eliminating the $b_i$'s, we get an isomorphism of $K$-algebras 
$$
\QH(X)/(a_1) \simeq A/(a_1) \simeq K[a_2]/(a_2^{n-1}).
$$ 
In particular, we have a surjective homomorphism of $K$-algebras 
\begin{align}\label{Eq.: ring homo origin}
A \to K[a_2]/(a_2^{n-1}).
\end{align}
Thus, the dimension of $A$ is at least $n-1$. In fact, below we will see that $\dim_K A=n-1$.

\smallskip

Now let us examine the structure of $B$, i.e. we need to study
solutions of \eqref{Eq.: System for QH(IG(2,2n))} different from the
origin. It is convenient to rewrite \eqref{Eq.: Relations for small QH II} as 
\begin{align*}
(z^4 + (2a_2 - a_1^2) z^2 + a_2^2)(z^{2n-4} + b_1 z^{2n-6} + \cdots +
  b_{n-2}) = z^{2n} - a_1,
\end{align*}
where we set $q = 1$ for convenience.
By making the substitution $a_1 = z_1 + z_2, a_2 = z_1 z_2$, and putting $Q(z)=z^{2n-4} + b_1 z^{2n-6} + \cdots + b_{n-2}$ we arrive at
\begin{align}\label{Eq.: equation in z}
(z^2 -z_1^2)(z^2 - z_2^2) Q(z) = z^{2n} - (z_1 + z_2).
\end{align}
In geometric terms this manipulation corresponds to pulling back our system \eqref{Eq.: System for QH(IG(2,2n))} with respect to the morphism
\begin{align*}
& \spec (K[z_1,z_2, b_1, \dots, b_{n-2}]) \to \spec (K[a_1,a_2, b_1, \dots, b_{n-2}]) \\
\intertext{defined by}
& a_1 \mapsto z_1 +z_2, \quad   a_2 \mapsto z_1z_2 ,   \quad    b_i \mapsto b_i.
\end{align*}
It is a double cover unramified outside of the locus $z_1=z_2$.

Let us count solutions of \eqref{Eq.: equation in z} for which $z_1 \neq z_2$ and both of them are non-zero. This reduces to counting pairs $z_1, z_2$ satisfying
\begin{align}\label{Eq.: Aux Eq}
& z_1^{2n}=z_1+z_2 \\ \notag
& z_2^{2n}=z_1+z_2.
\end{align}
Eliminating $z_2$ using the first equation we obtain that $z_1$ must be a solution of
\begin{align*}
(z_1^{2n}-z_1)^{2n}= z_1^{2n}.
\end{align*}
Now it is straightforward to count that there are $(2n-2)(2n-1)$ distinct pairs $(z_1, z_2)$, with $z_1 \neq z_2 \neq 0$, satisfying the system \eqref{Eq.: Aux Eq}.

In terms of the original system \eqref{Eq.: System for QH(IG(2,2n))} the above computation means that there are at least $(n-1)(2n-1)$ distinct solutions outside of the origin. Note that we have divided the number of solutions by 2, since the initial count was on the double cover. In other words the dimension of $B$ is at least $(n-1)(2n-1)$.

\smallskip

Up to now we have shown that $\dim_K (A) \geq n-1$ and  $\dim_K (B) \geq (n-1)(2n-1)$. Since $\dim_K(\QH(X)) = 2n(n-1)$, this implies that $\dim_K (A) = n-1$ and  $\dim_K (B) = (n-1)(2n-1)$. Hence, \eqref{Eq.: ring homo origin} is an isomorphism.  
\end{proof}

\begin{rmk}
The above proposition implies that $\QH(\IG(2,2n))$ is not semisimple, as was already mentioned in the introduction. Moreover, we have an explicit description of the non-semisimple factor as $K[\varepsilon]/(\varepsilon^{n-1})$. Note that the latter is the Milnor algebra of the $A_{n-1}$-singularity.
\end{rmk}

\section{Four-point Gromov--Witten invariants}
\label{sec:4points}

In Section \ref{Sec.: BQH} we will study the deformation of $\QH(\IG(2,2n))$ in the big quantum cohomology $\BQH(\IG(2,2n))$ in the direction of the special Schubert class $\sigma_2$. In particular, we will need the values of Gromov--Witten invariants
$$\langle \pt,\sigma_2,\sigma_i,\sigma_j \rangle_1,$$
which we compute in this section. Note that for dimension reasons these invariants vanish unless $i + j = 2n-2$. 

Morally a GW invariant $\langle \D_1, \dots , \D_n \rangle_d$ counts the number of rational curves of degree $d$ on $X$ incident to general representatives of the cohomology classes $\D_i$. Even though for a general variety $X$ this enumerative meaning fails, in our case $X=\IG(2,2n)$ we have the following fact.

\begin{lemma}\label{Lemma: enumerative meaning}
The Gromov--Witten invariant $\langle \pt,\sigma_2,\sigma_i,\sigma_j \rangle_1$ is the number of lines meeting general representatives of the cohomology classes $\pt$, $\sigma_2$, $\sigma_i$ and $\sigma_j$. Furthermore, given any open dense subset of the set of lines, the above lines can be chosen in this open subset.
\end{lemma}

\begin{proof}
It follows from Kleiman--Bertini's Theorem \cite{K} (see also \cite[Lemma 14]{FuPa}).
\end{proof}

\subsection{Lines on $\IG(2,2n)$} 

We recall the description of the Hilbert scheme of lines on $\IG(2,2n)$ from \cite[Theorem 4.3]{MaLa} and \cite[Proposition 3]{strickland}. First start with the Hilbert scheme of lines on the ordinary Grassmannian $\G(2,2n)$. The Hilbert scheme of lines on a variety $X$ and the universal line are denoted $F_1(X)$ and $Z_1(X)$ respectively.

A line on $\G(2,2n)$ is given by a pair $(W_1,W_3)$ of nested subspaces of $\C^{2n}$ of dimension $1$ and $3$ respectively.  Geometrically the corresponding line is given by 
\begin{align}\label{Eq.: Nested subspaces}
\ell(W_1,W_3) = \{ V_2 \in \G(2,2n) \ | \ W_1 \subset V_2 \subset W_3 \}.
\end{align}
Extending the above description to families one can show that there exist isomorphisms 
\begin{align*}
& F_1(\G(2,2n)) \simeq \Fl(1,3;2n) \\
& Z_1(\G(2,2n)) \simeq \Fl(1,2,3;2n).
\end{align*}
Moreover, the natural projections from $Z_1(\G(2,2n)) \subset \G(2,2n) \times F_1(\G(2,2n))$ to $\G(2,2n)$ and $F_1(\G(2,2n))$ are given by forgetting a part of the flag
\begin{align*}
\xymatrix{
\Fl(1,2,3;2n) \ar[r]^<<<<{f} \ar[d]_p  & \G(2,2n) \\
\Fl(1,3;2n) 
}
\end{align*}
Now if the line $\ell(W_1,W_3)$ is contained in $\IG(2,2n)$, we have that any two dimensional subspace $V_2$ with $W_1 \subset V_2 \subset W_3$ is isotropic. This is easily seen to be equivalent to the fact that $W_1 \subset W_3^\perp$ or equivalently $W_1 \subset \ker(\omega\vert_{W_3})$ (note that in \cite[Proposition 3]{strickland}, this last condition is missing). In particular we have
$$F_1(\IG(2,2n)) = \{ \ell(W_1,W_3) \in F_1(\G(2,2n)) \ | \ \omega(W_1,W_3) = 0\}.$$
Schematically, this can be described as the zero locus of $\omega : \mathcal{W}_1 \otimes \mathcal{W}_3 \to \cO_{F_1(\G(2,2n))}$ in $F_1(\G(2,2n))$, where $\mathcal{W}_1$ and $\mathcal{W}_3$ are the tautological subbundles. In particular a general element $\ell(W_1,W_3) \in F_1(\IG(2,2n))$ is determined by $W_3$ since $W_1 = \ker(\omega\vert_{W_3})$. We make this more precise. Let $\G(3,2n)$ be the Grassmannian of $3$-dimensional vector subspaces in $\C^{2n}$ and $\IG(3,2n)$ the closed subset representing subspaces isotropic for the form $\omega$.

\begin{prop}
The Hilbert scheme $F_1(\IG(2,2n))$ is the blow-up of $\IG(3,2n)$ in $\G(3,2n)$.
\end{prop}

\begin{proof}
  Since $F_1(\IG(2,2n))$ is a subscheme of $F_1(\G(2,2n)) = \Fl(1,3;2n)$, we have a morphism $\pi : F_1(\IG(2,2n)) \to \G(3,2n)$ defined by $(W_1,W_3) \mapsto W_3$ by forgetting $W_1$. We prove that this is the desired blow-up. Denote by $\mathcal{W}_1$ and $\mathcal{W}_3$ the tautological rank $1$ and $3$ bundles on $\Fl(1,3;2n)$ and $\G(3,2n)$. Recall that $\IG(3,2n)$ is defined by the vanishing of $\omega : \Lambda^2\mathcal{W}_3 \to \cO_{\G(3,2n)}$. Pulling back via $\pi$ and using the fact that $\mathcal{W}_1 \subset \ker(\omega\vert_{\mathcal{W}_3})$, we get that $\pi^*\omega$ factors as follows:
  $$\xymatrix{\Lambda^2\mathcal{W}_3 \ar[r] \ar[dr]_-{\pi^*\omega} & \Lambda^2(\mathcal{W}_3/\mathcal{W}_1) \ar[d]^-{s} \\
    & \cO_{F_1(\IG(2,2n))}.}$$
Since the top map is a morphism of vector bundles and never vanishes, the vanishing of $\pi^*\omega$ is equivalent to the vanishing of $s$. This just means that the transform of the ideal defining $\IG(2,2n)$ in $\G(2,2n)$ to $F_1(\IG(2,2n))$ is $\Lambda^2(\mathcal{W}_3/\mathcal{W}_1)$ and therefore invertible. By the universal property of blow-up, we deduce that $\pi$ factors through the blow-up $p : Bl_{\IG(3,2n)}(\G(3,2n))$ of $\IG(3,2n)$ in $\G(3,2n)$:
  $$\xymatrix{F_1(\IG(2,2n)) \ar[r]^-{f} \ar[dr]_-{\pi} & Bl_{\IG(3,2n)}(\G(3,2n)) \ar[d]^-{p} \\
    & \G(3,2n).}$$
Since $\IG(3,2n)$ and $\G(3,2n)$ are smooth, we only need to check that $f$ is bijective since Zariski's Main Theorem will imply that $f$ is an isomorphism. Note that the all diagram is $\Sp_{2n}$-equivariant and is bijective on the complement of $\IG(3,2n)$: the subspace $W_1$ is determined by $W_1 = \ker(\omega\vert_{W_3})$. Now the fiber $\pi^{-1}(W_3)$ for $W_3 \in \IG(3,2n)$ is isomorphic to $\p(W_3) \simeq \p^2$ and is homogeneous under the stabiliser of $W_3$ in $\Sp_{2n}$. Since the fiber of $p$ over $W_3$ is also isomorphic to $\p^2$ (the codimension of $\IG(3,2n)$ in $\G(3,2n)$ is $3$), the map $f$ must be bijective on the fiber. The result follows.
\end{proof}

The group of symplectic automorphisms $\Sp_{2n}$ acts on $\IG(2,2n)$ and $F_1( \IG(2,2n))$. In terms of \eqref{Eq.: Nested subspaces} there are two types of lines on $\IG(2,2n)$ corresponding to the two $\Sp_{2n}$-orbits on $F_1(\IG(2,2n))$:
\begin{itemize}
\item[(i)] If $W_3$ is not isotropic, then $W_1$ is the kernel of the symplectic form $\omega$ on $W_3$. Therefore, $W_1$ is completely determined by $W_3$. A line of this type will be denoted by $\ell(W_3)$. These lines form the open orbit of the $\Sp_{2n}$-action.  
\item[(ii)] If $W_3$ is isotropic, then $W_1$ is any one-dimensional subspace of $W_3$. These lines form the closed orbit of the $\Sp_{2n}$-action.
\end{itemize}
Note that in terms of the blow-up description the open orbit is $\pi^{-1}(\G(3,2n) \setminus \IG(3,2n))$ and the closed orbit is the exceptional divisor of the blow-up which is isomorphic to the isotropic flag variety $\textrm{IFl}(1,3;2n)$.

\subsection{Computation of the invariant}

Let $X$ be $\IG(2,2n)$. Denote by $Y$ the Hilbert scheme of lines on $X$ described above, and by $\Yo$ the open orbit of the $\Sp_{2n}$-action. In order to compute $\langle \pt,\sigma_2,\sigma_i,\sigma_j \rangle_1$, according to Lemma \ref{Lemma: enumerative meaning}, it is enough to consider lines in the open orbit $\Yo$.

For a closed subvariety $Z \subset X$ we define
$$
\Yo(Z) = \{ \ell \in \Yo \ | \ \ell \cap Z \neq \emptyset \},
$$
which is a closed subvariety of $\Yo$.

\begin{lemma}
(i) There exist identifications 
\begin{align*}
& \Yo(\{E_2\}) \overset{def}{=} \{ \ell(W_3) \in \Yo \ | \ E_2 \subset W_3\} \simeq \p(\bC^{2n}/E_2) \setminus \p(E_2^\perp/E_2) \\
& \Yo(Z(E_{2n - k -1})) = \{ \ell(W_3) \in \Yo \ | \ \dim(W_3 \cap E_{2n - k - 1}) \geq 1 \},
\end{align*}
where $\{E_2\}$ is the one-point subvariety of $X$ corresponding to an isotropic subspace $E_2 \subset \bC^{2n}$, and $Z(E_{2n - k -1})$ was defined in \eqref{Eq.: Schubert subvariety}.

\smallskip

(ii) If $E_2$ and $E_{2n - k - 1}$ are in general position, the intersection  $\Yo(\{E_2\}) \cap \Yo(Z(E_{2n - k -1}))$ is isomorphic to
$$\p \bigg( \big(E_2 + E_{2n - k - 1} \big) /E_2\bigg) \setminus \p \bigg( \big( (E_2 + E_{2n - k - 1}) \cap E_2^\perp \big)/E_2 \bigg),$$
which we view as a subvariety of $\p(\bC^{2n}/E_2) \setminus \p(E_2^\perp/E_2)$.
\end{lemma}

\begin{proof} 

(i) The first isomorphism just maps $W_3$ to $W_3/E_2$ and uses the fact that $W_3$ is non-isotropic iff $W_3/E_2$ is not contained in $E_2^\perp/E_2$. 

\smallskip

The second equality works as follows. Consider a point $V_2 \in (\ell(W_3) \cap Z(E_{2n - k -1})) \subset X$. Then, by definition of $\ell(W_3)$ and  $Z(E_{2n - k -1})$, we have $V_2 \subset W_3$ and $\dim(V_2 \cap E_{2n - k - 1}) \geq 1$. Thus, we get $\dim(W_3 \cap E_{2n - k - 1}) \geq 1$. Conversely, for $W_3$ satisfying $\dim(W_3 \cap E_{2n - k - 1}) \geq 1$, one can find a $2$-dimensional isotropic subspace $V_2 \subset W_3$ meeting $E_{2n - k - 1}$ non-trivially.

\smallskip

(ii) First, assume that $k = 0$. Then, we have $E_2 + E_{2n - 1} = \bC^{2n}$, as $E_2$ and $E_{2n - 1}$ are assumed to be in general position. Further, since we clearly have $\Yo(Z(E_{2n -1})) = \Yo$, the claim follows.

\smallskip

Now assume that $k \geq 1$. Since  $E_2$ and $E_{2n -k- 1}$ are in general position, we have $E_2 \cap E_{2n - k - 1} = 0$. If $\ell(W_3)$ is a point of $\Yo(\{E_2\}) \cap \Yo(Z(E_{2n - k -1}))$, then we have the inclusion $W_3 \subset E_2 + E_{2n - k - 1}$. Therefore, we have $W_3/E_2 \subset (E_2 + E_{2n - k - 1})/E_2$, which proves the claim. 
\end{proof}

\begin{cor}\label{Cor.: 4-point invariant}
 $\langle \pt,\sigma_2,\sigma_i,\sigma_j \rangle_1 = \delta_{i + j,2n-2}.$
\end{cor}

\begin{proof}

As was mentioned at the beginning of this section, the invariant vanishes unless $i + j = 2n-2$. So we assume $i + j = 2n-2$. In that case, $\langle \pt,\sigma_2,\sigma_i,\sigma_j \rangle_1$ is the number of lines meeting $\{E_2\}$, $Z(E_{2n - 3})$, $Z(E_{2n - i - 1})$ and $Z(E_{2n - j - 1})$, where the subspaces $E_k \subset \bC^{2n}$ are in general position with $\omega\vert_{E_k}$ of minimal rank. 

The locus of such lines is the intersection in $\p(\bC^{2n}/E_2) \setminus \p(E_2^\perp/E_2)$ of three linear spaces
$$\p \big((E_2 + E_{2n - 3})/E_2 \big), \ \ \p \big((E_2 + E_{2n - i - 1})/E_2 \big) \ \ \textrm{ and } \ \  \p \big((E_2 + E_{2n - j - 1})/E_2 \big).$$
Since these linear spaces are in general position, and of respective codimensions $1$, $i-1$ and $j-1$ that add up to the dimension of $\p(\bC^{2n}/E_2)$,  they meet in exactly one point.
\end{proof}

\section{Big quantum cohomology of $\IG(2,2n)$}
\label{Sec.: BQH}

In this section we show the regularity of $\BQH(\IG(2,2n))$ and deduce its generic semisimplicity, which was proved previously in \cite{GMS,Pe}.

\subsection{Deformation}

As before we let $X=\IG(2,2n)$ and consider the deformation $\BQH_\tau(X)$ of the small quantum cohomology $\QH(X)$ inside the big quantum cohomology $\BQH(X)$ in the direction of the Schubert class $\tau = \sigma_2$. Thus, the ring $\BQH_\tau(X)$ is the quotient of $\BQH(X)$ with respect to the ideal generated by those elements of $\big(H^*(X,\bQ)\big)^*$ which vanish on $\sigma_2$. 

Explicitly $\BQH_\tau(X)$ can be described in the following way. For any cohomology classes $a,b \in H^*(X)$ the product is of the form
$$ a \star_\tau b = a \starz b + t \sum_{\sigma} \bigg( \sum_{d \geq 1}  \langle a,b,\sigma,\sigma_2 \rangle_d q^d \bigg) \sigma^\vee + O(t^2),$$
where $a \starz b$ is the small quantum product, $\sigma$ runs over the basis of the cohomology consisting of Schubert classes, $\sigma^\vee$ is the dual basis, and $t$ is the deformation parameter of degree $-1$.

Let us have a closer look at the products of the form $\sigma_i \star_\tau \sigma_j$. According to the dimension axiom for GW invariants $\langle \sigma_i,\sigma_j,\sigma,\sigma_2 \rangle_d $ vanishes unless 
$$
i + j + \deg(\sigma) + 2 =  (2n-1)d + 2(2n-2).
$$ 
Therefore, applying Corollary \ref{Cor.: 4-point invariant}, we have 
\begin{align}\label{Eq.: BQH product for special classes}
\sigma_i \star_\tau \sigma_j = \sigma_i \starz \sigma_j +  \delta_{i+j,2n-2} qt + O(t^2),
\end{align}
for $i+j \leq 2n-2$ .

\subsection{Presentation for $\BQH_\tau(X)$}

Consider the following elements in $\BQH_\tau(X)$
$$\begin{array}{l}
\Delta_r = \det(\sigma_{1 + j - i})_{1 \leq i,j \leq r} \ \ \ \ \ \ \ \textrm{for $r \in [3,2n-2]$} \\
\\
\displaystyle{\Sigma_{2n-2} = \sigma_{n-1} \star_\tau \sigma_{n-1} + 2 \sum_{i = 1}^{n - 1} (-1)^i \sigma_{n - 1 + i} \star_\tau \sigma_{n - 1 - i}} \\
\\
\displaystyle{\Sigma_{2n} = \sigma_n \star_\tau \sigma_n + 2 \sum_{i = 1}^{n - 2} (-1)^i \sigma_{n + i} \star_\tau \sigma_{n - i} + (-1)^{n+1} q \sigma_{1},}\\
\end{array}
$$
where all products are taken in $\BQH_\tau(X)$. Note that these are ``the same" elements as the relations in the  presentation of $\QH(X)$ in Theorem \ref{thm-bkt}. The lower indices indicate the degree of the respective element in $\BQH_\tau(X)$.

\begin{rmk}
The variable $t$ can only appear together with a positive power of $q$ so that $q^dt$ can only occur in elements of degree at least $\deg(q) + \deg(t) = 2n-2$.
\end{rmk}

\begin{lemma}
\label{Lemma-Presentation-BQH}
We have $\Delta_r = O(t^2)$ for all $r \in [3,2n-2]$,  $\Sigma_{2n-2}
= (-1)^n qt + O(t^2)$ and $\Sigma_{2n} = O(t)$.
\end{lemma}

\begin{proof}

(i) Here we prove the statement about $\Delta_r$'s. From Theorem \ref{thm-bkt} and the above remark we know that $\Delta_r = O(t^2)$ for $r \in [3,2n-3]$. Thus, we only need to consider $\Delta_{2n-2}$. Inductively developing the determinant with respect to the first column we have
\begin{align*}
&\Delta_{2n-2} = \sum_{s=1}^{2n-2} (-1)^{s-1} \sigma_s \star_\tau \Delta_{2n-2-s} = O(t^2) - \sigma_{2n-4} \star_\tau \Delta_2 + \sigma_{2n-3} \star_\tau \Delta_1 - \sigma_{2n-2},
\end{align*}
where we used the fact that $\Delta_r = O(t^2)$ for $r \in [3,2n-3]$. Thus, we need to prove that 
\begin{align}\label{Eq.: formula I in the proof of Lemma O(t^2)}
\sigma_{2n-4} \star_\tau \sigma_2  - \sigma_{2n-4} \star_\tau \sigma_1 \star_\tau \sigma_1 + \sigma_{2n-3} \star_\tau \sigma_1  - \sigma_{2n-2}
\end{align}
is of the form $O(t^2)$. We will see this by reducing everything to products of special Schubert classes. First, let us look at the term $\sigma_{2n-4} \star_\tau \sigma_1 \star_\tau \sigma_1$. Since 
$\sigma_{2n-4} \star_0 \sigma_1 =  \sigma_{2n-3} + \sigma_{2n-3}',$
where $\sigma_{2n-3}'$ is a Schubert class of degree $2n-3$ different from $\sigma_{2n-3}$ (see \cite{ChPe}), and for degree reasons, we have
\begin{align}\label{Eq.: formula II in the proof of Lemma O(t^2)}
(\sigma_{2n-4} \star_\tau \sigma_1 ) \star_\tau \sigma_1 = ( \sigma_{2n-3} + \sigma_{2n-3}' ) \star_\tau \sigma_1.
\end{align}
By \eqref{Eq.: BQH product for special classes} we have $\sigma_{2n-3} \star_\tau \sigma_1 =  \sigma_{2n-3} \starz \sigma_1 + qt + O(t^2)$. Thus, we only need to take care of the term  $\sigma_{2n-3}' \star_\tau \sigma_1$. For degree reasons we have
\begin{align*}
&\sigma_{2n-3}' \star_\tau \sigma_1 = \sigma_{2n-3}' \starz \sigma_1 +\langle \sigma_{2n-3}',\sigma_1,\pt ,\sigma_2 \rangle_1 qt + O(t^2).            \end{align*}
Applying the dimension axiom for Gromov--Witten invariants we see that the 4-point invariant $ \langle \sigma_{2n-3}',\sigma_1,\pt ,\sigma_2 \rangle_1$ is equal to the 3-point invariant $\langle \sigma_{2n-3}',\pt ,\sigma_2 \rangle_1$. The latter can be computed using results in \cite{ChPe} and we get $\langle \sigma_{2n-3}',\pt ,\sigma_2 \rangle_1 = 1$. Therefore, we obtain
\begin{align}\label{Eq.: formula III in the proof of Lemma O(t^2)}
&\sigma_{2n-3}' \star_\tau \sigma_1 = \sigma_{2n-3}' \starz \sigma_1 + qt + O(t^2).            
\end{align}
Plugging \eqref{Eq.: formula II in the proof of Lemma O(t^2)}--\eqref{Eq.: formula III in the proof of Lemma O(t^2)} into \eqref{Eq.: formula I in the proof of Lemma O(t^2)}, using \eqref{Eq.: BQH product for special classes} and the fact that $\Delta_{2n-2}$ vanishes modulo $t$, we obtain the required statement.

\medskip

(ii) To prove the statement about $\Sigma_{2n-2}$ we just apply \eqref{Eq.: BQH product for special classes} to each summand of $\Sigma_{2n-2}$, and use and the fact that $\Sigma_{2n-2}$ vanishes modulo $t$.

\medskip

(iii) The statement about $\Sigma_{2n}$ is evident, as it just claims that there might be a non-trivial deformation along $t$.
\end{proof}

By a standard argument (see e.g. \cite[Proposition~10]{FuPa}) we obtain the following.

\begin{cor}\label{Presenation for BQHtau}
The ring $\BQH_\tau(X)$ is the quotient of $(K[[t]])[\sigma_1,\cdots,\sigma_{2n-2}]$ modulo relations of the form
\begin{align}\label{Eq.: Rel 1}
& \Delta_r + O(t^2)  \quad \quad   \text{for all} \quad r \in [3,2n-2] \\ \label{Eq.: Rel 2}
& \Sigma_{2n-2} + (-1)^{n+1} qt + O(t^2) \\ \label{Eq.: Rel 3}
& \Sigma_{2n} + O(t).
\end{align}
\end{cor}

\subsection{Structure of $\BQH_\tau(X)$}

Recall from Section \ref{SubSec.: Structure of QH} that for $\QH(X)$ we have the direct product decomposition 
\begin{align*}
& \QH(X)= A \times B = A \times \prod_{i \in I} B_i
\end{align*}
where $A=K[\varepsilon]/(\varepsilon^{n-1})$ each $B_i$ is the ground
field $K$. By Hensel's lemma this decomposition lifts to $\BQH_\tau(X)$ and we have
\begin{align}\label{Eq.: Direct product decomposition for BQH_tau}
& \BQH_\tau(X) = \mathcal{A} \times  \prod_{i \in I} \mathcal{B}_i
\end{align}
where $\mathcal{A}$ and $\mathcal{B}_i$ are algebras over $K[[t]]$. 

\smallskip

By construction the $K[[t]]$-algebra $\BQH_\tau(X)$ is a free module over $K[[t]]$ of finite rank equal to $\dim H^*(X)$. Therefore, $\cA$ and $\cB_i$'s are also free $K[[t]]$-modules of finite rank. Reducing modulo $t$ one sees that the rank of $\cA$ is $n-1$ and the ranks of $\cB_i$'s are all equal to 1. Similarly we obtain that the natural algebra homomorphism $K[[t]] \to \cB_i$ is an isomorphism of $K$-algebras.

\begin{thm}\label{Prop.: Regularity of BQH}
The ring $\BQH(X)$ is regular.
\end{thm}

\begin{proof}

Let us first prove that $\BQH_\tau(X)$ is a regular ring, which amounts to showing that $\cA$ is regular. Consider the presentation of $\BQH_\tau(X)$ given in Corollary \ref{Presenation for BQHtau}. Taking linear terms one computes the Zariski tangent space of $\cA$ at the maximal ideal (it corresponds to the origin in coordinates~$\sigma_1, \dots, \sigma_{2n-2}, t$). We have: \eqref{Eq.: Rel 1} contribute $\sigma_r$ for $r \in [3,2n-2]$, the relation \eqref{Eq.: Rel 2} contributes
$2 \sigma_{2n-2} + qt$, and the relation \eqref{Eq.: Rel 3} contributes $\sigma_1$ (note that the deformation along $t$ has not changed the linear part of this relation). Therefore, the dimension of the tangent space is 1. Since the Krull dimension of $\cA$ is also equal to 1, we obtain regularity.  

The regularity of $\BQH(X)$ follows from an identical argument. The only difference is the number of the deformation parameters. Note that no additional GW invariants are necessary.
\end{proof}

\begin{cor}\label{Thm.: semisimplicity}
$\BQH(X)$ is generically semisimple.
\end{cor}

\begin{proof}
This is a simple corollary of the above Theorem \ref{Prop.: Regularity of BQH}. Let us give a short proof for completeness. The coordinate ring of the generic fiber 
\begin{align*}
\BQH(X)_{\eta} = \BQH(X) \otimes_{K[[t_0, \dots, t_s]]} K((t_0, \dots, t_s))
\end{align*}
is a localization of $\BQH(X)$. Therefore, it is also a regular ring, since $\BQH(X)$ was regular. This implies that $\BQH(X)_{\eta}$ is a product of finite field extensions of $K((t_0, \dots, t_s))$ which was our definition of semisimplicity. 
\end{proof}

\section{Relation to the unfolding of $A_{n-1}$-singularity}
\label{Sec.: Singularity}

In this section we establish a relation between the quantum cohomology of $\IG(2,2n)$ and the unfolding of an $A_{n-1}$-singularity. The comparison takes place at the level of $F$-manifolds. The concept of an $F$-manifold was introduced in \cite{HeMa} as a weakening of that of a Frobenius manifold. Every Frobenius manifold gives rise to an $F$-manifold after forgetting a part of the structure. Both $F$-manifolds and Frobenius manifolds arise naturally in quantum cohomology and singularity theory. We begin by briefly recalling relevant concepts and refer to \cite{He} for details.

\subsection{$F$-manifolds: generalities}
\label{SubSec.: F-mfds generalities}

Let $M$ be a complex manifold. Endow $\cT_M$ (the holomorphic tangent sheaf) with an $\cO_M$-bilinear, commutative, associative, and unital product $\circ \colon \cT_M \times \cT_M \to \cT_M$, and denote $e$ the global unit vector field. The triple $(M, \, \circ \, ,\, e )$ is called an \textit{$F$-manifold}, if the identity
\begin{align}\label{Eq.: Integrability condition}
\Lie_{X \circ Y} (\circ) = X \circ \Lie_Y(\circ) + Y \circ \Lie_X(\circ)
\end{align}
is satisfied for any local vector fields $X$ and $Y$.

Given two $F$-manifolds $(M_1, \, \circ_1 \, ,\, e_1)$ and $(M_2, \, \circ_2 \, ,\, e_2)$ one can endow the product $M_1 \times M_2$ with a natural structure of an $F$-manifold (see Proposition 2.10 of \cite{He}). The resulting $F$-manifold is denoted $(M_1 \times M_2, \circ_1 \boxplus \circ_2 , e_1 \boxplus e_2)$.

If $(M, \, \circ \, ,\, e )$ is an $F$-manifold and $P \in M$ a point, then by $(M, P,  \, \circ \, ,\, e )$ we will denote the germ of $(M, \, \circ \, ,\, e )$ at $P$. A germ $(M, P,  \, \circ \, ,\, e )$ is called \textit{irreducible} if it cannot be written as a product of germs of $F$-manifolds. 

Let $(M, P,  \, \circ \, ,\, e )$ be a germ of an $F$-manifold. The tangent space $T_PM$ is a finite dimensional commutative $\bC$-algebra, and as such it uniquely decomposes into the product of local algebras $T_PM = B_1 \times \dots \times B_r$. Moreover, this decomposition extends to an analytic neighbourhood of $P$, i.e. we have the decomposition 
\begin{align*}
(M, P,  \, \circ \, ,\, e )= \prod_{i=1}^r (M_i, P_i,  \, \circ_i \, ,\, e_i ) 
\end{align*}
of germs of irreducible $F$-manifolds (see Theorem 2.11 of \cite{He}).

\subsection{$F$-manifold of $\BQH(X)$}
\label{SubSec.: F-mfd BQH}

In Section \ref{Sec.: Conventions and Notation} we have defined big quantum cohomology ring $\BQH(X)$ as an algebra over $K[[t_0, \dots , t_s]]$, where $K$ is the algebraic closure of $\bQ((q))$. Another way of packaging the ingredients appearing in $\BQH(X)$ is in the language of formal Frobenius manifold and formal F-manifolds (see \cite{Ma}). 

For reasons that will become clear later, the setting of formal $F$-manifolds is not suitable for our needs. Therefore, we will describe how to pass to the analytic setting by making some convergence assumption. As was already mentioned in Remark \ref{SubSubSec.: Remark on difference of notation with Manin's book}, $\BQH(X)$ can be defined as an algebra over $R[[t_0, \dots , t_s]]$, where $R=\bQ[[q]]$. Moreover, since $X$ is a Fano variety, one can further replace $R$ with $\bQ[q]$. Further, we can specialize to $q=1$ and replace $\bQ$ with $\bC$. This way we get $\BQH(X)$ defined over $\bC[[t_0, \dots , t_s]]$ which we will denote $\BQH(X)_{\bC}$.

Another way to formulate the above procedure is to say that for a Fano variety $X$, if we put $q=1$, the Gromov--Witten Potential \eqref{Eq.: GW potential} gives an element of the power series ring $\bC[[t_0, \dots , t_n]]$. Let us denote this power series by $\Phi_{q=1}$.

A priori it is not clear that $\Phi_{q=1}$ has a non-trivial domain of convergence. This is one of the standard expectations in Gromov--Witten theory, but by no means an established result. Therefore, we will make the following assumption.

\smallskip

\noindent \textbf{Convergence assumption:} the power series $\Phi_{q=1}$ converges in some open neighbourhood $M$ of the origin. 

\smallskip

Repeating the definition of quantum multiplication we now obtain an analytic $F$-manifold (even Frobenius) structure on $M$. Namely, basis elements $\Delta_i$ appearing in \eqref{Eq.: Quantum cohomology} form a basis of sections of the holomorphic tangent sheaf of $M$. Thus, Formula \eqref{Eq.: Quantum cohomology} endows $\cT_M$ with a multiplication, and one can further check that \eqref{Eq.: Integrability condition} is satisfied.

\subsection{$F$-manifolds for isolated hypersurfaces singularities}

An {\it isolated hypersurface singularity} is a holomorphic function germ $f \colon (\bC^m,0) \to (\bC,0)$ with an isolated singularity at the origin, i.e. the origin is an isolated critical point of $f$. One defines the {\it Milnor algebra} of $f$ as $\cO_{\bC^m,0}/(\frac{\partial f}{\partial x_1} , \dots , \frac{\partial f}{\partial x_m})$, which is the algebra of functions on the critical locus of $f$. 

An {\it unfolding} of an isolated hypersurface singularity $f \colon (\bC^m,0) \to (\bC,0)$ is a holomorphic function germ $F \colon (\bC^m,0) \times (\bC^n,0)   \to (\bC,0)$ satisfying $F_{|\bC^m \times \{0\}} = f$. The germ $(\bC^n,0)$ is called the {\it base} of the unfolding and is often denoted $(M,0)$. The {\it relative critical locus} $C \subset (\bC^m,0) \times (M,0)$ is the closed subspace defined by the ideal $J_F= (\frac{\partial F}{\partial x_1} , \dots , \frac{\partial F}{\partial x_m})$. Note that $C$ is not necessarily reduced. This setup can be conveniently described by the diagram
\begin{align}\label{Eq.: Unfolding Diagram}
\xymatrix{
C \ar[r]^<<<<i \ar[rd]_p &(\bC^m \times M,0)  \ar[d]^{\pi} \ar[r]^<<<<<F & (\bC,0)  \\
&(M,0)
}
\end{align}
Define a morphism of sheaves 
\begin{align}\label{Eq.: KS}
\cT_{(M,0)} \to p_*\cO_C
\end{align} 
by sending a local vector field $X$ to $\widetilde{X}(F)_{|C}$, where $\widetilde{X}$ ist a lift of $X$ to $(\bC^m \times M,0)$. One says that unfolding $F$ is {\it semiuniversal} if \eqref{Eq.: KS} is an isomorphism (cf. \cite[Theorem 5.1]{He}). In this case one can transport the product structure from $p_*\cO_C$ to $\cT_M$. Moreover, this multiplication endows $(M,0)$ with an $F$-manifold structure (cf. \cite[Theorem 5.3]{He}).

\subsubsection{Example: $A_n$-singularity}

Let us illustrate the above construction in the simplest case of $f=x^{n+1}$, which is also the relevant case for our considerations. In this case a semiuniversal unfolding is given by
\begin{align*}
& F= x^{n+1} + a_{n-1} x^{n-1} + \dots + a_1x+a_0,
\end{align*}
where $a_0, \dots , a_{n-1}$ are coordinates on the base space of the unfolding $M=\bC^n$. The fact that $f$ is a function of one variable is a special feature of this example. Among isolated hypersurface singularities it is the only one that can be realised in this way.

\subsection{Spectral cover of an $F$-manifold}

Let $(M, \, \circ \, ,\, e )$ be an $F$-manifold. The \textit{spectral cover} of $(M, \, \circ \, ,\, e )$ is defined as the relative analytic spectrum $\Specan(\cT_M, \circ)$. By construction there exists a canonical morphism $\Specan(\cT_M, \circ) \to M$. In general $\Specan(\cT_M, \circ)$ is only an analytic space, i.e. it could be singular and/or non-reduced. The canonical morphism $\Specan(\cT_M, \circ) \to M$ is 
finite, and in the reduced case can be seen as a ramified covering of~$M$.

In the quantum cohomology setting the spectral cover is given by the spectrum of the big quantum cohomology ring. In the singularity theory setting the spectral cover is nothing else but the relative critical locus $C$ that appears in \eqref{Eq.: Unfolding Diagram}.

An $F$-manifold $(M, \, \circ \, ,\, e )$ is called generically semisimple if there is an open subset $U \subset M$, so that for every point in $P \in U$ the algebra structure on $T_PM$ is semisimple.

One can nicely characterise generic semisimplicity of an $F$-manifold in terms of its spectral cover. Namely, we have the following
\begin{thm}[\cite{He}, Theorem 3.2]
An $F$-manifold $(M, \, \circ \, ,\, e )$ is generically semisimple if and only if its spectral cover is reduced.
\end{thm}
Note that, the ``if'' part of the statement is an immediate consequence of the generic smoothness in characteristic zero.

The next theorem can be viewed as a refinement of the above one and gives a beautiful characterisation of (germs of) $F$-manifolds with smooth spectral covers.

\begin{thm}[\cite{He}, Theorem 5.6]\label{Thm: smoothness}
Any germ of an irreducible generically semisimple $F$-manifold with smooth spectral cover is isomorphic to the base space of a semiuniversal unfolding of an isolated hypersurface singularity.   
\end{thm}

\subsection{$F$-manifold of $\BQH(\IG(2,2n))$}
\label{SubSec.: F-mfd IG(2,2n)}

Let $(M, O,  \circ, e) $ be the germ at the origin of the $F$-manifold for the big quantum cohomology of $\IG(2,2n)$ as defined in Section \ref{SubSec.: F-mfd BQH}. The decomposition of $\QH(\IG(2,2n))$ described in Section \ref{SubSec.: Structure of QH} gives rise to the decomposition of the algebra $(T_OM, \circ_O) = \QH(\IG(2,2n))_{\bC}$ into the product of local algebras. Thus, according to Section \ref{SubSec.: F-mfds generalities}, we obtain the product decomposition
\begin{align*}
(M, O,  \circ, e) = (M_0, P_0, \circ_0, e_0) \times \prod_{i=1}^{(2n-1)(n-1)} (M_i , P_i,  \circ_i, e_i), 
\end{align*}
where $(M_i , P_i,  \circ_i, e_i)$ are one-dimensional $F$-manifolds and  $(M_0 , P_0,  \circ_0, e_0)$ is an irreducible $(n-1)$-dimensional $F$-manifold.

Computing the rank of the Jacobian matrix, similar to the proof of Theorem \ref{Prop.: Regularity of BQH}, implies that the spectral cover of $(M, O,  \circ, e)$ and, hence, of $(M_i , P_i,  \circ_i, e_i)$ is smooth. Thus, by Theorem \ref{Thm: smoothness} each germ $(M_i , P_i,  \circ_i, e_i)$ is isomorphic to the base space of a semiuniversal unfolding of an isolated hypersurface singularity. Moreover, we can identify the type of this singularity.

\begin{thm}\label{Thm: determination of singularity}

If the convergence assumption of Section \ref{SubSec.: F-mfd BQH} is satisfied for $\IG(2,2n)$, then the following statements hold.

\smallskip

(i) The $F$-manifold $(M_0, P_0, \circ_0, e_0)$ is isomorphic to the base space of a semiuniversal unfolding of an isolated hypersurface singularity of type $A_{n-1}$. 

\smallskip

(ii) For $i \geq 1$ the $F$-manifolds $(M_i , P_i,  \circ_i, e_i)$  are isomorphic to the base space of a semiuniversal unfolding of an isolated hypersurface singularity of type $A_1$.
\end{thm}

\begin{proof}
We only prove the first assertion. The second one is proved in a similar way. 

The statement follows from the following observation. Proposition \ref{Prop.: Decomposition of QH} implies that the tangent space to $\Spec T_{P_0}M_0$ at its unique point is of dimension one. Now, assuming that $(M_0, P_0, \circ_0, e_0)$ is given by an unfolding of $f$, we get an identification between the Milnor algebra of $f$ and $T_{P_0}M_0$. In singularity theory the dimension of the tangent space to the spectrum of the Milnor algebra is known as {\it corank of $f$ at the critical point}. Now, the statement follows from the fact that the only singularity of corank 1 and Milnor number $n-1$ is the $A_{n-1}$-singularity (see Section 11.1 of \cite{AGV1}).
\end{proof}

\begin{rmk}
Note that the convergence assumption is classically expected (for example, see~\cite{CoIr} and references therein). Further, it may be possible that Theorem~\ref{Thm: smoothness} holds in the formal setting~\cite{HeMail}. 
\end{rmk}

\section{LG model: expectations}
\label{Sec.: LG model}

In this paragraph we will see how our previous results fit into the general picture of homological mirror symmetry for Fano varieties. Note that we are not going to prove a mirror symmetry type statement, but rather speculate about properties of a Landau--Ginzburg mirror for $\IG(2,2n)$.

\subsection{Mirror symmetry for Fano varieties}

Let $X$ be a smooth projective Fano variety and $(Y,w)$ a Landau--Ginzburg model, i.e. $Y$ is a smooth variety and $w$ a regular function on it. To be able to formulate mirror-symmetry-type statements we also endow $X$ and $Y$ with symplectic forms compatible with respective complex structures but our notation will not reflect the symplectic forms explicitly. For simplicity we assume that $w$ has only isolated, but possibly degenerate, critical points.

To $X$ one attaches two triangulated categories --- the category of D-branes of type A and the category of D-branes of type B. The category of D-branes of type A is the derived Fukaya category $D \Fuk(X)$, which encodes the symplectic geometry of $X$ and is independent of the complex structure. The category of D-branes of type B is  $D^b \Coh (X)$  --- the bounded derived category of coherent sheaves on $X$, which encodes the complex geometry of $X$ and is independent of the symplectic structure.

Further one defines categories of D-branes of type A and B for the pair $(Y,w)$. The category of D-branes of type A is the derived Fukaya--Seidel category $D \FS(Y, w)$ and the category of D-branes of type B is the triangulated category of matrix factorisations~$\MF(Y ,w)$. As before the category of D-branes of type A is independent of the complex structure and the category of D-branes of type B of the symplectic one.

One says that $X$ and $(Y, w)$ are mirror to each other if there exist triangulated equivalences of categories
\begin{align}\label{Eq.: Mirror equivalence I}
& D  \Fuk(X) \simeq \MF(Y , w) \\ \label{Eq.: Mirror equivalence II}
& D^b \Coh (X) \simeq D \FS(Y, w).
\end{align}

As it is not essential for our purposes, we refrain from recalling precise definitions of the above categories and instead refer to \cite{Or} and references therein. 

\subsection{Quantum cohomology and mirror symmetry}

Up to now we viewed the quantum cohomology of $X$ as a purely algebro-geometric object. Historically, Gromov--Witten invariants and quantum cohomology were first defined in symplectic geometry \cite{RuTi}, and in mirror symmetry they appear on the symplectic side. Namely, conjecturally there exists an isomorphism of algebras
\begin{align*}
\QH(X) \to HH^*(D \Fuk(X)),
\end{align*}
known as closed-open map (e.g.~see \cite{Ga}), from the small quantum cohomology to the Hochschild cohomology of $D \Fuk(X)$.

Let us assume that $X$ and $(Y,w)$ are mirror to each other. Then \eqref{Eq.: Mirror equivalence I} gives rise to the isomorphism
\begin{align*}
HH^*(D \Fuk(X)) \simeq HH^*(\MF(Y,w)).
\end{align*}
According to \cite[Theorem 3.1]{LinPo}, the Hochschild cohomology $HH^*(\MF(Y,w))$ is isomorphic to the Milnor algebra of $(Y,w)$. Therefore, reduced points of $\Spec \QH(X)$ correspond to non-degenerate critical points of $w$, whereas fat points correspond to degenerate ones.

\subsection{LG model for $\IG(2,2n)$}
\label{SubSec.: SOD}

Let us assume that $\IG(2,2n)$ has a mirror LG model $(Y,w)$ with only isolated critical points. Then the above discussion combined with Proposition~\ref{Prop.: Decomposition of QH} imply that the critical locus of $w$ has only one degenerate critical point. Moreover, Theorem~\ref{Thm: determination of singularity} predicts that the degenerate critical point is of type $A_{n-1}$.

The above should imply (see \cite{Or,Sei}) that $D \FS(Y, w)$ has a semiorthogonal decomposition of the form 
\begin{align}\label{Eq.: FS decomposition}
D \FS(Y, w) = \langle \cC_{n-1}, E_1 , \dots , E_{(2n-1)(n-1)}  \rangle
\end{align} 
where $E_i$ are exceptional objects given by vanishing thimbles attached to non-degenerate critical points and $\cC_{n-1}$ is the subcategory generated by vanishing thimbles attached to the degenerate one.

Further, it should also follow that subcategory $\cC_{n-1}$ is the Fukaya--Seidel category of an isolated hypersurface singularity of type $A_{n-1}$. The latter was studied by P.~Seidel in \cite{Sei} and shown to be equivalent to the bounded derived category of representations of the quiver of type $A_{n-1}$. Note that $\cC_{n-1}$ also has a full exceptional collection. 

The above discussion combined with equivalence \eqref{Eq.: Mirror equivalence II} suggest that $D^b \Coh (\IG(2,2n))$ should have a semiorthogonal decomposition (in fact, a full exceptional collection) 
\begin{align*}
D^b \Coh (\IG(2,2n)) = \langle \cA_{n-1}, E_1 , \dots , E_{(2n-1)(n-1)}  \rangle,
\end{align*} 
where subcategory $\cA_{n-1}$ is equivalent to the bounded derived category of representations of the quiver of type $A_{n-1}$. Such a decomposition is constructed explicitly in Theorem \ref{theorem:residual-igr} of the Appendix by Alexander Kuznetsov.

\newpage

\section{Appendix by Alexander Kuznetsov}
\label{Sec.: Appendix}

\subsection{Lefschetz exceptional collections}

Recall the notion of a Lefschetz semiorthogonal decomposition (\cite[Def.~4.1]{Ku07} and \cite[Def.~3.1]{Ku14}).
In this paper we will need its special case --- a Lefschetz exceptional collection (\cite[Def.~2.1]{Ku}). 
We will remind this notion for the readers convenience.

\begin{defi}
A {\sf Lefschetz exceptional collection} in $D^b(X)$ with respect to the line bundle $\cO_X(1)$ is an exceptional collection in $D^b(X)$ which has a block structure
\begin{equation*}
\left(
\underbrace{E_1 , E_2, \dots, E_{\lambda_0}}_{\text{block 1}}, 
\underbrace{E_1(1) , E_2(1), \dots, E_{\lambda_1}(1)}_{\text{block 2}}, 
\dots,
\underbrace{E_1(m-1) , E_2(m-1), \dots, E_{\lambda_m}(m-1)}_{\text{block $m$}}
\right)
\end{equation*}
where $\lambda = (\lambda_0 \ge \lambda_1 \ge \dots \ge \lambda_{m-1} \ge 0$) is a non-increasing sequence of positive integers (the {\sf support partition} of the Lefschetz collection).
A Lefschetz exceptional collection is {\sf full} if it generates the derived category $D^b(X)$.
It is called {\sf rectangular} if $\lambda_0 = \lambda_1 = \dots = \lambda_m$ (i.e., if the Young diagram of its support partition is a rectangle).
\end{defi}

Clearly, to specify a Lefschetz exceptional collection one needs to give its first block and its support partition.
We will denote by $(E_\bullet,\lambda)$ the corresponding Lefschetz exceptional collection.
Let us list several interesting examples.

\begin{example}
Let $X = \G(2,m)$. Set $k = \lfloor m/2 \rfloor$, $E_i = S^{i-1}\cU^*$ for $1 \le i \le k$, and
\begin{equation*}
\lambda = 
\begin{cases}
(k^{2k+1}),
& \text{if $m = 2k + 1$ is odd,}\\
(k^k,(k-1)^k),
& \text{if $m = 2k$ is even} 
\end{cases}
\end{equation*}
(here the exponents stand for the multiplicities of the entries, thus the first line denotes the partition with $2k+1$ parts equal to $k$, and the second line means the partition with $k$ parts equal $k$ and $k$ parts equal $k-1$).
\end{example}

\begin{example}
Let $X = \IG(2,2k)$. Set $E_i = S^{i-1}\cU^*$ for $1 \le i \le k$, and $\lambda = (k^{k-1},(k-1)^k)$
with the same convention about the exponents as in the previous example.
\end{example}

The collections listed in these examples are full \cite[Theorems 4.1, 5.1]{Ku}.
With the conventions we took they can be rewritten as
\begin{align}
\label{eq:gr2-2kp1}%
D^b(\G(2,2k+1)) &= \langle S^{\bullet-1}\cU^*,(k^{2k+1}) \rangle,\\
\label{eq:gr2-2k}%
D^b(\G(2,2k)) &= \langle S^{\bullet-1}\cU^*,(k^k,(k-1)^k) \rangle,\\
\label{eq:igr2-2k}%
D^b(\IG(2,2k)) &= \langle S^{\bullet-1}\cU^*,(k^{k-1},(k-1)^k) \rangle.
\end{align}
The first of these collections is rectangular, while the last two are not.

It is known that $\G(k,n)$ has a rectangular Lefschetz exceptional collection of length $n$ if and only if $k$ and $n$ are coprime \cite{Fo}.
Lefschetz exceptional collections on general isotropic Grassmannians are not yet known, it is a good question to construct such a collection.

\begin{defi}
Given a Lefschetz exceptional collection $(E_\bullet,\lambda)$ with the length of $\lambda$ equal to $m$,
the subcollection 
\begin{equation*}
(E_\bullet,(\lambda_m)^m) = 
\left(
E_1, \dots, E_{\lambda_m}, 
E_1(1), \dots, E_{\lambda_m}(1), 
\dots,
E_1(m-1), \dots, E_{\lambda_m}(m-1)
\right)
\end{equation*}
is called {\sf the rectangular part} of $(E_\bullet,\lambda)$.
The subcategory orthogonal to the rectangular part of a Lefschetz exceptional collection is called the {\sf residual subcategory}.
It is zero if and only if the original collection is full and rectangular.
\end{defi}

The main result of this section is the following two Theorems describing the residual subcategories of~\eqref{eq:gr2-2k} and~\eqref{eq:igr2-2k}. Note that the residual subcategory of~\eqref{eq:gr2-2kp1} is zero.

\begin{thm}\label{theorem:residual-gr}
The residual subcategory of~\eqref{eq:gr2-2k} is generated by $k$ completely orthogonal exceptional objects.
In particular, it is equivalent to the derived category of the union of $k$ disjoint points.
\end{thm}

\begin{thm}\label{theorem:residual-igr}
The residual subcategory of~\eqref{eq:igr2-2k} is equivalent to the derived category of representations of $A_{k-1}$ quiver.
\end{thm}

We prove Theorem~\ref{theorem:residual-gr} in Section~\ref{subsection:usual}, and Theorem~\ref{theorem:residual-igr} in Section~\ref{subsection:isotropic} below.

\subsection{Usual Grassmannian}\label{subsection:usual}

Denote the block of the rectangular part of the collection~\eqref{eq:gr2-2k} by $\cA$, so that
\begin{equation*}
\cA = \langle \cO, \cU^*, \dots, S^{k-2}\cU^* \rangle.
\end{equation*}
The Lefschetz decomposition of $\G(2,2k)$ then can be rewritten as
\begin{multline}\label{eq:lefschetz-gr-explicit}
D^b(\G(2,2k)) = \langle 
\cA, S^{k-1}\cU^*, 
\cA(1), S^{k-1}\cU^*(1), 
\dots,
\cA(k-1), S^{k-1}\cU^*(k-1), \\
\cA(k), \cA(k+1),  \dots, \cA(2k-1) 
\rangle.
\end{multline}
We denote the corresponding residual subcategory of $\G(2,2k)$ by $\cR_k$, so that
\begin{equation*}
\cR_k = \langle \cA, \cA(1), \dots, \cA(2k-1) \rangle^\perp.
\end{equation*}
Using mutations of exceptional collections it is easy to deduce from~\eqref{eq:lefschetz-gr-explicit} that
\begin{equation}\label{eq:rk-1}
\cR_k = \langle 
\LL_{\cA}(S^{k-1}\cU^*), 
\LL_{\langle \cA,\cA(1) \rangle}(S^{k-1}\cU^*(1)), 
\dots
\LL_{\langle \cA,\cA(1),\dots,\cA(k-1) \rangle}(S^{k-1}\cU^*(k-1))
\rangle,
\end{equation}
where $\LL_\cA$ is the left mutation through $\cA$, $\LL_{\langle \cA,\cA(1) \rangle}$ is the left mutation through $\langle \cA, \cA(1) \rangle$, and so on.
So, it is enough to show that these objects (they are automatically exceptional) are completely orthogonal.

For this we will use an alternative description of the objects, obtained by using the following long exact sequence (constructed in~\cite[Proof of Lemma~4.3]{Ku}):
\begin{multline}\label{eq:complex-gr}
\scriptstyle
0 \to
S^{k-1}\cU^*(-k) \to
V \otimes S^{k-2}\cU^*(1-k) \to
\ldots \to
\Lambda^{k-2}V \otimes \cU^*(-2) \to
\Lambda^{k-1}V \otimes \cO_X(-1) \to
\\ \scriptstyle \to
\Lambda^{k-1}V^* \otimes \cO_X \to
\Lambda^{k-2}V^* \otimes \cU^* \to
\ldots \to
V^* \otimes S^{k-2}\cU^* \to
S^{k-1}\cU^* \to 0.
\end{multline}
Its second line is the Koszul complex, and its cohomology is isomorphic to $\Lambda^{k-1}\cU^\perp$.
Its first line is the dual of the second line twisted by $\cO_X(-1)$ taking into account isomorphisms $S^i\cU \cong S^i\cU^*(-i)$ (obtained by taking the symmetric powers of the isomorphism $\cU \cong \cU^*(-1)$). Consequently, the cohomology of the first line is isomorphic to $\Lambda^{k-1}(V/\cU) \otimes \cO_X(-1)$.
So, the canonical isomorphism
\begin{equation*}
\Lambda^{k-1}\cU^\perp \cong \Lambda^{k-1}(V/\cU) \otimes \cO_X(-1)
\end{equation*}
allows to glue the two lines into a single long exact sequence.

We denote by $F_i$ the cohomology of the subcomplex of~\eqref{eq:complex-gr} consisting of its first $i$ terms.
In other words, $F_i$ are defined by the following exact sequences
\begin{equation}\label{eq:f-left}
0 \to
S^{k-1}\cU^*(-k) \to
V \otimes S^{k-2}\cU^*(1-k) \to
\ldots \to
\Lambda^{i-1}V \otimes S^{k-i}\cU^*(i-k-1) \to
F_i \to 0.
\end{equation}

\begin{lemma}\label{lemma:gr-fi}
For every $1 \le i \le k$ there is an isomorphism 
\begin{equation*}
\LL_{\langle \cA,\dots,\cA(k-i) \rangle}(S^{k-1}\cU^*(k-i)) \cong F_{i}(k-i)[2k-i - 1].
\end{equation*}
\end{lemma}
\begin{proof}
Since~\eqref{eq:complex-gr} is exact, we have yet another exact sequence for $F_i$:
\begin{multline}\label{eq:f-right}
0 \to F_i \to \Lambda^iV \otimes S^{k-i-1}\cU^*(i-k) \to \dots \to 
\Lambda^{k-1}V \otimes \cO_X(-1) \to
\\ \to
\Lambda^{k-1}V^* \otimes \cO_X \to
\Lambda^{k-2}V^* \otimes \cU^* \to
\ldots \to
V^* \otimes S^{k-2}\cU^* \to
S^{k-1}\cU^* \to 0.
\end{multline}
Twisting it by $\cO_X(k-i)$ we note that all its terms (except for the leftmost $F_i(k-i)$ and the rightmost $S^{k-1}\cU^*(k-i)$) lie in the subcategory $\langle \cA, \dots, \cA(k-i-1), \cA(k-i) \rangle$.
In other words, the cone of the morphism 
\begin{equation*}
S^{k-1}\cU^*(k-i) \to F_i(k-i)[2k-i-1]
\end{equation*}
represented by the extension class of this exact sequence, is contained in the subcategory $\langle \cA, \dots, \cA(k-i-1), \cA(k-i) \rangle$.
So, to prove the claim of the lemma, it is enough to show that $F_i(k-i) \in \langle \cA, \dots \cA(k-i) \rangle^\perp$.
For this we use the twist by $\cO_X(k-i)$ of~\eqref{eq:f-left}.
The terms of this twist clearly belong to the first line of~\eqref{eq:lefschetz-gr-explicit} twisted by $\cO_X(-k)$, 
hence to the orthogonal of the second half of~\eqref{eq:lefschetz-gr-explicit} twisted by $\cO_X(-k)$, i.e., to $\langle \cA, \cA(1), \dots, \cA(k-1) \rangle^\perp$.
This gives the required embedding.
\end{proof}

Combining this lemma with~\eqref{eq:rk-1} we deduce an exceptional collection
\begin{equation}\label{eq:rk-2}
\cR_k = \langle F_k, F_{k-1}(1), \dots, F_1(k-1) \rangle.
\end{equation} 
To prove the Theorem it remains to show it is completely orthogonal.

\begin{lemma}\label{lemma:gr-hom-fi}
We have $\Ext^\bullet(F_i(k-i),F_j(k-j)) = 0$ for all $1 \le j < i \le k$.
\end{lemma}
\begin{proof}
To show this we use for $F_i(k-i)$ resolution~\eqref{eq:f-right} twisted by $\cO_X(k-i)$,
and for $F_j(k-j)$ resolution~\eqref{eq:f-left} twisted by $\cO_X(k-j)$.
The terms of the first are contained in the subcategory
\begin{equation*}
\langle \cA, \dots, \cA(k-i-1), \cA(k-i), S^{k-1}\cU^*(k-i) \rangle,
\end{equation*}
and the terms of the second are contained in the subcategory
\begin{equation*}
\langle S^{k-1}\cU^*(-j), \cA(1-j), \dots, \cA(-1) \rangle.
\end{equation*}
All the components of these subcategories are clearly semiorthogonal --- this follows immediately from the twist of~\eqref{eq:lefschetz-gr-explicit} by $\cO_X(-j)$,
since $-j < k-i \le k - 1 - j$ by the assumption on $i$ and $j$ taken in the lemma.
\end{proof}

This Lemma together with the exceptional collection~\eqref{eq:rk-2} proves Theorem~\ref{theorem:residual-gr}.

\subsection{Isotropic Grassmannian}\label{subsection:isotropic}

Denote the block of the rectangular part of the collection~\eqref{eq:igr2-2k} by $\cIA$, so that
\begin{equation*}
\cIA = \langle \cO, \cU^*, \dots, S^{k-2}\cU^* \rangle.
\end{equation*}
The Lefschetz decomposition of $\IG(2,2k)$ then can be rewritten as
\begin{multline}\label{eq:lefschetz-igr-explicit}
D^b(\IG(2,2k)) = \langle 
\cIA, S^{k-1}\cU^*, 
\cIA(1), S^{k-1}\cU^*(1), 
\dots,
\cIA(k-2), S^{k-1}\cU^*(k-2), \cIA(k-1), \\
\cIA(k),  \dots, \cIA(2k-2) 
\rangle.
\end{multline}
We denote the corresponding residual subcategory of $\IG(2,2k)$ by $\cIR_k$, so that
\begin{equation*}
\cIR_k = \langle \cIA, \cIA(1), \dots, \cIA(2k-2) \rangle^\perp.
\end{equation*}
Using mutations of exceptional collections it is easy to deduce from~\eqref{eq:lefschetz-igr-explicit} that
\begin{equation}\label{eq:bar-rk}
\cIR_k = \langle 
\LL_{\cIA}(S^{k-1}\cU^*), 
\LL_{\langle \cIA,\cIA(1) \rangle}(S^{k-1}\cU^*(1)), 
\dots
\LL_{\langle \cIA,\cIA(1),\dots,\cIA(k-2) \rangle}(S^{k-1}\cU^*(k-2))
\rangle.
\end{equation}
The first step of the proof in this case is to show that these objects are given by the (restrictions to $\IG(2,2k)$) of the same sheaves $F_i$ as before.

\begin{lemma}\label{lemma:igr-fi}
For every $2 \le i \le k$ there is an isomorphism
\begin{equation*}
\LL_{\langle \cIA,\dots,\cIA(k-i) \rangle}(S^{k-1}\cU^*(k-i)) \cong F_{i}(k-i)[2k-i - 1].
\end{equation*}
\end{lemma}
\begin{proof}
The same proof as in Lemma~\ref{lemma:gr-fi} applies. 
We restrict exact sequences~\eqref{eq:f-left} and~\eqref{eq:f-right} to $X = \IG(2,2k)$.
Twisting the restriction of~\eqref{eq:f-right} by $\cO_X(k-i)$ we note that all its terms (except for the leftmost $F_i(k-i)$ and the rightmost $S^{k-1}\cU^*(k-i)$) lie in the subcategory $\langle \cIA, \dots, \cIA(k-i-1), \cIA(k-i) \rangle$.
In other words, the cone of the morphism 
\begin{equation*}\label{eq:f-s}
S^{k-1}\cU^*(k-i) \to F_i(k-i)[2k-i-1]
\end{equation*}
represented by the extension class of this exact sequence, is contained in the subcategory $\langle \cIA, \dots, \cIA(k-i-1), \cIA(k-i) \rangle$.
So, to prove the claim of the lemma, it is enough to show that $F_i(k-i) \in \langle \cIA, \dots \cIA(k-i) \rangle^\perp$.
For this we use the twist by $\cO_X(k-i)$ of the restriction of~\eqref{eq:f-left}.
The terms of this twist belong to the first line of~\eqref{eq:lefschetz-igr-explicit} twisted by $\cO_X(-k)$, 
hence to the orthogonal of the second line of~\eqref{eq:lefschetz-igr-explicit} twisted by $\cO_X(-k)$, i.e., to $\langle \cIA, \cIA(1), \dots, \cIA(k-2) \rangle^\perp$.
This gives the required embedding.
\end{proof}

Combining this lemma with~\eqref{eq:bar-rk} we deduce an exceptional collection
\begin{equation}\label{eq:bar-rk-2}
\cIR_k = \langle F_k, F_{k-1}(1), \dots, F_2(k-2) \rangle.
\end{equation} 

\begin{rmk}
Note that the object $F_1(k-1) = S^{k-1}\cU^*(-1)$ also belongs to $\cIR_k$.
This follows easily by mutating the first component $\cIA$ of~\eqref{eq:lefschetz-igr-explicit} to the right (it gets twisted by the anticanonical class $\cO_X(2k-1)$), and then twisting the resulting decomposition by $\cO_X(-1)$.
\end{rmk}

To prove the Theorem it remains to compute $\Ext$-spaces between its objects.

\begin{lemma}\label{lemma:gr-hom-fi-2}
Assume $1 \le j < i \le k$.
We have 
\begin{equation*}
\Ext^\bullet(F_i(k-i),F_j(k-j)) = 
\begin{cases}
\C, & \text{if $i = j + 1$},\\
0, & \text{otherwise.}
\end{cases}
\end{equation*}
\end{lemma}
\begin{proof}
To show this we use for $F_i(k-i)$ (the restriction of) resolution~\eqref{eq:f-right} twisted by $\cO_X(k-i)$,
and for $F_j(k-j)$ (the restriction of) resolution~\eqref{eq:f-left} twisted by $\cO_X(k-j)$.
The terms of the first are contained in the subcategory
\begin{equation*}
\langle \cIA, \dots, \cIA(k-i-1), \cIA(k-i), S^{k-1}\cU^*(k-i) \rangle,
\end{equation*}
and the terms of the second are contained in the subcategory
\begin{equation*}
\langle S^{k-1}\cU^*(-j), \cIA(1-j), \dots, \cIA(-1) \rangle.
\end{equation*}
In case of $i \ge j + 2$ all the components of these subcategories are clearly semiorthogonal --- this follows immediately from the twist of~\eqref{eq:lefschetz-igr-explicit} by $\cO_X(-j)$,
since $-j < k-i \le k - 2 - j$ by the assumption on $i$ and $j$.
If, however, $i = j + 1$ we do not have a semiorthogonality between $S^{k-1}\cU^*(k-i)$ and $S^{k-1}\cU^*(-j)$.
On a contrary, a direct Borel--Bott--Weil calculation shows that
\begin{equation*}
\begin{aligned}
\Ext^\bullet(S^{k-1}\cU^*(k-i),S^{k-1}\cU^*(-j)) &\cong
\Ext^\bullet(S^{k-1}\cU^*,S^{k-1}\cU^*(i-j-k)) \\ &\cong
\Ext^\bullet(S^{k-1}\cU^*,S^{k-1}\cU^*(1-k)) \cong \C[3-2k].
\end{aligned}
\end{equation*}
Finally, using exact sequences~\eqref{eq:f-left} and~\eqref{eq:f-right} we deduce that
\begin{align*}
\Ext^\bullet(F_i(k-i),F_j(k-j)) 
& \cong \Ext^\bullet(S^{k-1}\cU^*(k-i)[i+1-2k],S^{k-1}\cU^*(-j)[j-1]) \\
& \cong \Ext^\bullet(S^{k-1}\cU^*(k-i),S^{k-1}\cU^*(-j))[2k-3] \cong \C.\qedhere
\end{align*}
\end{proof}

\smallskip

\begin{rmk}
A chain of morphisms $F_k \to F_{k-1}(1) \to \dots \to F_{2}(k-2) \to F_{1}(k-1)$ computed in Lemma~\textup{\ref{lemma:gr-hom-fi-2}} can be described as follows.
Consider the bicomplex of~\cite[Proposition~5.3]{Ku}. Replacing $m$ by $k$ and $W$ by $V$ \textup{(}according to the difference in the notation\textup{)}, 
and identifying $V$ with $V^*$ via the symplectic form, we note that the columns of the obtained bicomplex coincide with the sequences~\eqref{eq:f-left} twisted by $\cO(2k-i)$.
Since the totalization of the bicomplex is exact, we conclude that the cohomology sheaves of the vertical differential of the bicomplex form an exact sequence
\begin{equation*}
0 \to F_k(k) \to F_{k-1}(k-1) \to \dots \to F_2(2k-2) \to F_1(2k-1) \to 0,
\end{equation*}
from which the required chain can be obtained by a twist.
\end{rmk}

\smallskip

Lemma \ref{lemma:gr-hom-fi-2} together with the exceptional collection~\eqref{eq:rk-2} proves Theorem~\ref{theorem:residual-igr}.
Indeed, we can define the functor $D^b(A_{k-1}) \to \cIR_k$ by the rule 
\begin{equation*}
S_i \mapsto F_{k+1-i}(i-1)[1-i],
\qquad 1  \le i \le k-1
\end{equation*}
where $S_i$ stand for the simple representations of the quiver. Since
\begin{equation*}
\Ext^\bullet(S_i,S_j) = 
\begin{cases}
\C[-1], & \text{if $i = j - 1$},\\
0, & \text{otherwise,}
\end{cases}
\end{equation*}
it follows that the functor is fully faithful, and by~\eqref{eq:bar-rk-2} it is essentially surjective.

\begin{rmk}
Under this equivalence the extra object $F_1(k-1) \in \cIR_k$ corresponds (up to a shift) to the projective module of the first vertex of the quiver.
\end{rmk}

\end{document}